\documentclass[12pt]{amsart}
\usepackage{amsmath,amsfonts,amsthm,amscd,amssymb,mathrsfs,amssymb, multirow}
\usepackage{stmaryrd}
\usepackage{graphicx}
\usepackage{url}
\usepackage{mathrsfs}
\usepackage{tikz}
\newtheorem{thm}{Theorem}
\newtheorem{lem}[thm]{Lemma}
\newtheorem{deff}[thm]{Definition}

\newtheorem{prop}[thm]{Proposition}
\newtheorem{expl}[thm]{Example}
\newtheorem{nexpl}[thm]{Non-example}

\newtheorem{theorem}{Theorem}

\numberwithin{figure}{section}
\newcommand{\res}[1]{\left.#1\right|}
\newcommand{\overbar}[1]{\mkern 1.5mu\overline{\mkern-1.5mu#1\mkern-1.5mu}\mkern 1.5mu}

\newcommand{\R}{\mathbb{R}} 
\newcommand{\C}{\mathbb{C}}

\newcommand{\Norm}[1]{\left\Vert #1 \right\Vert}

\newcommand{\p}{\partial}             
\newcommand{\abs}[1]{\lvert#1 \rvert} 
\newcommand{\inn}[1]{\langle #1\rangle}
\newcommand{\innn}[1]{\langle \hspace{-2pt} \langle #1\rangle \hspace{-2pt} \rangle}

\newcommand{\ring}[1]{\overset{\circ}{#1}}
\numberwithin{equation}{section}
\numberwithin{thm}{section}
\numberwithin{theorem}{section}
\numberwithin{table}{section}
\numberwithin{table}{section}
\begin{document}
\title{The Calabi metric and desingularization of Einstein orbifolds}
\author{Peyman Morteza}
\address{Department of Mathematics, University of Wisconsin, Madison, 
WI, 53706}
\email{morteza@math.wisc.edu}
\author{Jeff A. Viaclovsky}
\address{Department of Mathematics, University of California, Irvine, 
CA, 92697}
\email{jviaclov@uci.edu}
\thanks{The authors were partially supported by NSF Grant DMS-1405725. Part of this work was completed while the authors were in residence at Mathematical Sciences Research Institute in Berkeley, California. The authors would like to thank MSRI for their support, and for providing such an excellent working environment}
\begin{abstract} 
Consider an Einstein orbifold $(M_0,g_0)$ of real dimension $2n$ having a singularity 
with orbifold group the cyclic group of order $n$ in ${\rm{SU}}(n)$ which is generated 
by an $n$th root of unity times the identity. Existence of a Ricci-flat K\"ahler ALE metric with this group at infinity 
was shown by Calabi. There is a natural ``approximate'' Einstein 
metric on the desingularization of $M_0$ obtained by replacing a small neighborhood of the singular point of the orbifold with a scaled and truncated Calabi metric.
In this paper, we identify the first obstruction to perturbing this 
approximate Einstein metric to an actual Einstein metric. 
If $M_0$ is compact, we can use this to produce examples of Einstein orbifolds which do not admit any Einstein metric in a neighborhood of the natural approximate Einstein metric
on the desingularization. 
In the case that $(M_0,g_0)$ is asymptotically hyperbolic Einstein and non-degenerate, 
we show that 
if the first obstruction vanishes, then there does in fact exist an asymptotically hyperbolic Einstein metric on the desingularization. We also obtain a non-existence result in 
the asymptotically hyperbolic Einstein case, provided that the obstruction does not vanish. This work extends a construction of Biquard in the case $n =2$, in which case the Calabi metric is also known as the Eguchi-Hanson metric, but there are some key points for which the higher-dimensional case differs. 
\end{abstract}
\maketitle
\setcounter{tocdepth}{1}
\vspace{-1cm}
\tableofcontents
\section{Introduction}
\label{intro}
Let $\Gamma_{n}=\{ e^{\frac{2k\pi i}{n}}: k=0,...,n-1\}$ and consider the action of $\Gamma_{n}$ on $\C^{n}$ by 
\begin{align}
\label{action}
e^{\frac{2k\pi i}{n}} \cdot(z^{1},...,z^{n})=(e^{\frac{2k\pi i}{n}}z^{1},...,e^{\frac{2k\pi i}{n}}z^{n}).
\end{align}
The Calabi metric $g_{cal}$ is a Ricci-flat K\"ahler asymptotically locally Euclidean (ALE) metric on the total space $X$ of the line bundle $\mathcal{O}(-n) \rightarrow \mathbb{P}^{n-1}$, 
 which is $\mathrm{U}(n)$-invariant. The group at infinity is $\Gamma_n \subset \mathrm{SU}(n)$. 

Let $(M_0,g_0)$ denote an Einstein orbifold of real dimension $m = 2n$ satisfying 
\begin{align}
     Ric(g_{0})=\Lambda g_{0},
\end{align}
and which has a singular point with group $\Gamma_n$. One can desingularize the orbifold metric by gluing on a Calabi metric in the following manner.

Since $g_{cal}$ is Ricci-flat and ALE, there exists a coordinate system $\{x^i\}, i = 1 \dots 2n$, 
and a diffeomorphism 
\begin{align}
\Psi : ( \mathbb{R}^{2n} \setminus B(0,R))/\Gamma_n \rightarrow X \setminus K,
\end{align}
where $B(0,R)$ is a ball of radius $R$, and $K$ is a compact subset of $X$ so that 
\begin{align}
\label{ALEorder}
\Psi^* g_{cal} &= g_{euc} + O(r^{-2n}), \\
\partial^\alpha \Psi^*g_{cal} &= O(r^{-2n-|\alpha|}), 
\end{align}
where $\alpha$ is a multi-index of order $|\alpha|$, $r = |x|$,
and $g_{euc}$ denotes the Euclidean metric.  
In the following 
we extend $r$ to a globally defined function on $X$ which 
satisfies $0 < r < R$ on the interior of the compact subset $K$.

Similarly, since $p_0$ is a smooth orbifold point of $(M_0,g_0)$, 
there exists a coordinate 
system $\{z^i\}, i = 1 \dots 2n$, and a diffeomorphism 
\begin{align}
\tilde{\Psi}:  B(0,\epsilon) \setminus \{0\} \rightarrow  \widetilde{B(p_0, \epsilon) \setminus \{p_0\} } 
\end{align}
such that 
\begin{align}
\tilde{\Psi}^* \pi^* g_0 &= g_{euc} + O(|z|^2), 
\end{align}
where $\pi : \widetilde{B(p_0, \epsilon)}\rightarrow B(p_0, \epsilon)$ is the universal covering
mapping, and such that $\tilde{\Psi}^* \pi^*(g) $ extends to a smooth metric on $B(0, \epsilon)$.

Let $t>0$ be small, and consider the following regions 
\begin{align}
X^t &= \{ x \in X :  r < 2 t^{-1/4} \}\\
M_0^t &= M \setminus B(p_0, (1/2) t^{1/4})
\end{align}
Define a mapping 
\begin{align}
\label{attachingmap}
\phi_{t}: \{ x \in X \ | \ (1/2) t^{-1/4} < r < 2 t^{-1/4} \} 
\rightarrow \{ m \in M_0 \ | \ (1/2) t^{1/4} < |z| < 2 t^{1/4} \}
\end{align}
by $\phi(x) = \sqrt{t}  x$. Define $M^t$ to be the
union of $X^t$ and $M_0^t$ along the attaching map $\phi_t$. 

\begin{lem} If $A \in \mathrm{SO}(2n)$ satisfies 
$A \gamma = \gamma A$ for $\gamma$ a generator of $\Gamma_n$, 
then $A \in \mathrm{U}(n)$, if $n \geq 3$. 
\end{lem}
This lemma is elementary, and the proof is omitted. It illustrates a 
key difference between the case $n=2$ and the higher-dimensional cases:
if $n=2$, then one can modify the attaching map $\phi$ by $2$ nontrivial rotational parameters, 
but there is no rotational freedom if $n > 2$ (since the Calabi metric is 
$\mathrm{U}(n)$-invariant). Another consequence 
of this lemma is that there is a well-defined almost 
complex structure $J$ at $p_0$. The corresponding K\"ahler form at $p_0$, 
will be denoted by $\omega_{p_0}$. 

An approximate Einstein metric on $M^t$ is defined as follows. 
Let $\chi$ be a smooth function satisfying $0 \leq \chi \leq 1$, 
\begin{align}
\label{cutoff}
\chi(s) = 
\begin{cases}
1 &  s \leq 1/2 \\
0 & s  \geq 2\\
\end{cases},
\end{align}
and define $\chi_{t}(s) \equiv \chi(t^{\frac{1}{4}}s)$. Then 
\begin{align}
g_t =
\begin{cases}
g_0  & M_0 \setminus B(p_0, 2 t^{1/4}) \\
(1 - \chi_{t}(r) )\phi_{t}^* g_0 + \chi_{t}(r) t \cdot g_{cal} & (1/2) t^{-1/4} < r < 2 t^{-1/4}\\
t \cdot g_{cal}  & r < (1/2) t^{-1/4} \\
\end{cases}
\end{align}
is a natural ``approximate'' Einstein metric on $M^t$.

\begin{figure}[h]
\begin{tikzpicture}
\draw[line width = .05 cm, lightgray]  (5,0) arc (3:177:5 cm and .5cm);
\draw[line width = .05 cm, black] (-5,0) arc (180:360:5 cm and .5cm);
\draw[line width = .05 cm, black](-5.02,-.01) -- (-1.2,2.3);
\draw[line width = .05 cm, lightgray] (2,2.5) arc (3:177:2 cm and .25cm);
\draw[line width = .05 cm, gray](-1.2,2.3) -- (0.02,3.01);
\draw[line width = .05 cm, gray](1.2,2.3) -- (-.02,3.01);
\draw[line width = .05 cm, black](1.2,2.3) -- (5.02,-.01);
\draw[line width = .05 cm, black] (-2,2.5) arc (180:360:2 cm and .25cm);
\draw[line width = .05 cm, black]  (-2.02,2.49) -- (-.5,3.5);
\draw[line width = .05 cm, black] (.5,3.5) -- (2.02,2.49);
\draw[line width = .05 cm, black] (-.5,3.5) to [out = 35, in = 145] (.5,3.5);
\draw[line width = .03 cm, black, ->] (-5.2,1.5) -- (-4,1);
\node[align = center, above, left] at (-5.2,1.5) {Einstein orbifold \\ $(M_0,g_{0})$};
\draw[line width = .03 cm, black, ->] (2.2,4) -- (1,3.5);
\node[align = center, above, right] at (2.2,4) {Calabi metric \\ $(X, t \cdot g_{cal})$};
\draw[dashed] (-2,4) -- (-2,-1);
\draw[dashed] (-.5,4) -- (-.5,-1);
\draw[dashed] (0,3) -- (0,-1);
\node[below] at (.1,-1.1) {$0$};
\draw[<-|] (-5,-1) -- (0, -1);
\node[align = left, below] at (-5,-1) {$|z|$};
\node[below] at (-2,-1) {$2 t^{\frac14}$};
\node[below] at (-.5,-1) {$\frac{1}{2} t^{\frac14}$};
\draw[<-|] (-5,4) -- (0, 4);
\node[align = left, below] at (-5,4) {$r$};
\node[above] at (-2,4) {$2 t^{-\frac14}$};
\node[above] at (-.5,4) {$\frac{1}{2} t^{-\frac14}$};
\node at (-1.25,1.25) {\footnotesize $|z| =r \sqrt{t}$};
\end{tikzpicture}
\caption{The desingularization procedure.}
\label{bubblefig}
\end{figure}
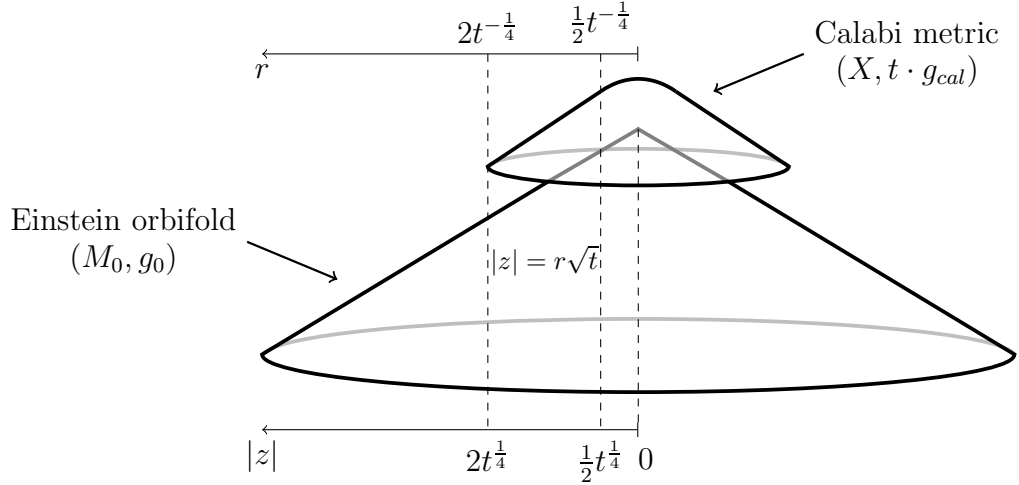
We define a weight function on $M^t$ as follows. 
\begin{align}
\label{weightdef}
w = 
\begin{cases}
1   & z \in M_0 \setminus B(p_0, i_{g_0})\\
|z| & \frac12 t^{\frac14}  <  |z| < \frac12 i_{g_0}\\
r t^{\frac12}  &   2 < r  <  2 t^{-\frac14}\\
t^{\frac12} &   r \leq 1 \\
\end{cases},
\end{align}
where $i_{g_0}$ is the injectivity radius of $g_0$, and 
such that $w$ is increasing in $r$ for $1 < r < 2$,
and increasing in $|z|$ for $ \frac12 i_{g_0} < |z| < i_{g_0}$.

Our first result is in the case of a compact Einstein orbifold, which is an extension of \cite[Theorem 1.2]{Biquard}. We define 
\begin{align}
\inn{\mathcal{R}(\omega),\omega}(p_0) = \sum_{i,j,k,l} R_{ijkl}(p_0)(\omega_{p_0})_{ij}
(\omega_{p_0})_{kl},
\end{align}
and let $R(p_0)$ denote the scalar curvature at $p_0$. 
\begin{theorem}
\label{theoremone}
Let $(M_0,g_0)$ be a compact Einstein orbifold having a singular point $p_0$ 
with orbifold group $\Gamma_n \subset {\rm{SU}}(n)$ with $n \geq 3$, 
and assume that 
\begin{align}
\label{obstruction}
n \inn{\mathcal{R}(\omega),\omega}(p_0) + 2(n-2) R(p_0) \neq 0.
\end{align}
Then there exist constants $\epsilon > 0$ and $\delta_0 > 0$ small, such that  
there does not exist any 
Einstein metric $g_E$ near $g_t$ satisfying
\begin{align}
\label{t1n}
\Vert  \nabla^{\ell} ( g_t - g_E) \Vert_{C^0(g_t)} < \epsilon w^{- \delta_0- \ell},
\end{align} 
for $0 \leq \ell \leq 3$, and sufficiently small gluing parameter $t$.
\end{theorem}
This theorem does give some new understanding of possible degenerations of Einstein metrics, but we do not obtain an existence result in the compact case because this gluing problem is {\textit{obstructed}}, and the criterion in \eqref{obstruction} is only the first obstruction. To find an Einstein metric in this manner, one would have to show the vanishing of infinitely many obstructions (both of a local and global nature). 
This would be difficult, if not impossible, to check in any given case. 
However, it is possible that this procedure could work in a specific example, an example of which is the following (and is a generalization of the well-known Kummer construction in complex dimension $2$). 
\begin{expl}{\rm
This example is due to Peng Lu \cite{PengLu} (and is proved using a gluing argument for K\"ahler-Einstein metrics).  
Let $\Sigma$ denote the elliptic curve with a non-trivial automorphism of order $3$. Consider the $6$-torus $\Sigma \times \Sigma \times \Sigma$. Let $\mathbb{Z}_3$ act diagonally on the product. Then $T^6 / \mathbb{Z}^3$ has $27$ singular points of type $\Gamma_3$. The flat metric on the quotient {\textit{can}} be resolved to a Calabi-Yau metric on the resolution by attaching a Calabi metric at each of the $27$ singular points. Clearly, the obstruction \eqref{obstruction} is satisfied at the singular points, since the orbifold metric is flat. 
}
\end{expl}

Theorem \ref{theoremone} implies several non-existence results. The first ``non-example'' is the following.
\begin{nexpl}{\rm
Let $(M,g)$ be a compact K\"ahler-Einstein orbifold with Einstein constant $\Lambda \neq 0$, 
and singular point $p_0$ of type $\Gamma_n$ and $n \geq 3$. In the K\"ahler-Einstein case, $\langle \mathcal{R}\omega, \omega \rangle = 2 R$ (see \cite[page 77]{Besse}), 
so the obstruction in \eqref{obst} is equal to $4(n-1) R(p_0)$. Consequently, if $\Lambda \neq 0$, then there does not exist any Einstein metric near the approximate metric $g_t$, for $t$ sufficiently small. We note that it follows from the adjunction formula that there is no K\"ahler-Einstein metric near $g_t$ for $t$ sufficiently small, so the main point is that these examples cannot be desingularized by a non-K\"ahler Einstein metric either. 
}
\end{nexpl}
 This ``non-example'' should be contrasted with the case of $n=2$, in which case there are known examples of sequences of positive K\"ahler-Einstein metrics limiting to orbifolds with Eguchi-Hanson metrics bubbling off \cite{Tian, OSS}. 
The ``explanation'' is that for $n=2$, the rotational parameters come into play, and the obstruction can be made to vanish upon rotating the complex structure, but this phenomenon does not happen in higher dimensions. 
\begin{nexpl}{\rm
The football metric $g_S$ on $S^{2n}/\Gamma$ is
defined as the quotient of the round metric on $S^{2n}$ by $\Gamma$, where $\Gamma$ acts
in the above manner on the first $2n$-coordinates of $\mathbb{R}^{2n+1}$. 
Since $g_S$ has constant curvature equal to $1$, the obstruction \eqref{obstruction}
does not vanish at either of the singular points. Consequently, there does not exist any Einstein metric near the approximate metric $g_t$, for $t$ sufficiently small. 
}
\end{nexpl}
We emphasize that Theorem \ref{theoremone} only can be used to rule out an Einstein metric near the specific approximate metric $g_t$ defined above. Of course, if the first obstruction in \eqref{obstruction} is non-zero, it is possible that there could be some other procedure to desingularize the orbifold points which could produce an Einstein metric, but the above approximate metric seems to be the most ``natural''. 
 
\subsection{Asymptotically hyperbolic Einstein metrics}
We next consider a class of complete, non-compact Riemannian metrics. 
\begin{deff}
{\em
Let $\overbar{M} = M \cup \p M$ be a smooth, compact manifold with boundary, and let $g$ be a Riemannian metric on $M$. We say that $g$ is {\textit{conformally compact}} if 
there exists a function
$x \in C^{\infty}(\overbar{M},\R)$ with $x>0$ on  $M$ such that
\begin{align} 
\label{deffunction}
\p M = \{ p\in \overbar{M}:x(p)=0\},
\end{align}	
and $\overbar{g}=x^{2}g$ extends to a Riemannian metric on $\overbar{M}$ of class $C^2$ 
with $dx\neq0$ on $\p M$.
}
\end{deff}
A function $x$ with the above properties is called a boundary defining function. 
Note that $\overbar{g}\vert_{\p M}$ is not well-defined since there might be another candidate for a boundary defining function $x$ but its conformal class $\lbrack\overbar{g}\vert_{\p M}\rbrack$ is well-defined and we call it the conformal infinity of $g$.
\begin{deff}
Let $(M,\p M,g,x)$ be a smooth, conformally compact manifold. We say that $(M,\p M,g,x)$ is asymptotically hyperbolic (AH) if $\abs{d x}_{\overbar{g}}=1$, where 
$x$ is a boundary defining function as above. 
If moreover $g$ is Einstein with $Ric(g) = - (m-1) g$, then $g$ is called 
asymptotically hyperbolic Einstein (AHE). 
\end{deff}
Note that the AH condition is independent of the defining function $x$. 
Furthermore, if $m = \dim(M)$ is even, the metric $\bar{g}$ extends smoothly 
to $\overbar{M}$ \cite{CDLS}. For more background on the theory of 
AHE metrics, we refer the reader to \cite{MR2260400, Lee, Leebook, Mazzeo1}.
We say that $g$ is non-degenerate if the space of $L^2$ infinitesimal Einstein 
deformations is trivial (see Definition \ref{deform}).  

Our main existence result is the following, which is an extension of \cite[Theorem~1.1]{Biquard} to the higher-dimensional case. 
\begin{theorem}
\label{theorem2}
Let $(M_0,g_0)$ be an asymptotically hyberbolic Einstein orbifold having a singular point $p_0$ with orbifold group $\Gamma_n \subset {\rm{SU}}(n)$, with 
conformal infinity $[ g_{\infty}]$ and $n \geq 3$. If $g_0$ is non-degenerate and
\begin{align}
\label{obst2}
n \inn{\mathcal{R}(\omega),\omega}(p_0) + 2(n-2) R(p_0) =0,
\end{align}
then there exists an AHE metric  $\hat{g}_t$ on $M^t$ for all $t$ sufficiently small. The conformal infinity $[\hat{g}_{\infty,t}]$ of the Einstein metric is a small perturbation of $[ g_{\infty}]$, in the sense that 
\begin{align}
\Vert \hat{g}_{\infty,t} -  g_{\infty} \Vert_{C^2(\partial M)} < C t. 
\end{align}
for some constant $C$.
\end{theorem}
The Einstein metric $\hat{g}_t$ is close, in a weighted norm,  
to a certain perturbation of~$g_t$ which changes the conformal infinity. 
For the precise statement, see Section~\ref{completionsec}. 
The only ``explicit'' example of this phenomenon that we know of is the following. 
\begin{expl}{\rm
In \cite{PagePope}, Page and Pope wrote down an explicit $1$-parameter family of 
${\rm{U}}(n)$-invariant Einstein metrics $\mathcal{O}(-n) \rightarrow \mathbb{P}^{n-1}$,
see also \cite{BB}. 
In \cite{MazzeoSinger}, it was pointed out by Mazzeo and Singer that
at one end of the parameter space, 
the Page-Pope family degenerates to an Einstein orbifold,
with a Ricci-flat ALE metric bubbling off (which is in fact the Calabi metric). 
It can be shown that the limiting orbifold 
is the quotient of a Page-Pope metric on $\mathbb{D}^{2n}$, which also comes in a 
$1$-dimensional family, and the parameter is 
exactly that for which the obstruction \eqref{obst2} is satisfied. Of course, Theorem \ref{theorem2} could be used to produce the Page-Pope metrics on $\mathcal{O}(-n)$, but since these metrics are already known explicitly, it is not necessary to analyze this example in more detail here.  
}
\end{expl}
We note that the boundary conformal class in the previous example is that of a Berger sphere. 
The boundary Berger sphere conformal classes in the Page-Pope family on $\mathcal{O}(-n)$ 
are those with smaller length of the Hopf fiber than the one for which 
the obstruction vanishes. The ones with longer fiber length probably do not have any AHE filling, but the proof of this would require a more detailed analysis than given here (in particular, it would require a computation of the second obstruction). 

There is also a corresponding non-existence result in the AHE case, provided that the obstruction in \eqref{obst2} does not vanish. But to state this properly, we need the following. Use the boundary defining function $x$ to identify 
a neighborhood of the boundary, call it $U_{2\epsilon}$ 
with $(0, 2\epsilon) \times \partial M_0$, and if 
necessary, we modify the defining function $x$ to a new
defining function $\tilde{x}$ satisfying
\begin{align}\label{txdef}
\tilde{x}(z) = 
\begin{cases}
x & z \in U_{2\epsilon}\\
2 \epsilon  \leq \tilde{x} \leq 1 & z \in M_0 \setminus \{ U_{2\epsilon}\}.\\
\end{cases}
\end{align}
Given a boundary metric $\tilde{g}_{\infty}$ close to $g_{\infty}$, 
using the identification above we 
extend $\tilde{g}_{\infty}$ to an asymptotically hyperbolic 
metric on $M_0$ by 
\begin{align}
\tilde{g} = g + \chi ( x \epsilon^{-1}) ( \tilde{g}_{\infty} - g_{\infty}).
\end{align}
Our non-existence result in the AHE case is the following. 
\begin{theorem} 
\label{theorem3}
Let $(M_0,g_0)$ be an asymptotically hyberbolic Einstein orbifold having a singular point $p_0$ with orbifold group $\Gamma_n \subset {\rm{SU}}(n)$, with 
conformal infinity $[ g_{\infty}]$ and $n \geq 3$. Assume that
\begin{align}
\label{obst3}
n \inn{\mathcal{R}(\omega),\omega}(p_0) + 2(n-2) R(p_0)  \neq 0.
\end{align}
Fix  $\delta_0 > 0 $ and $\delta_{\infty} > 0$ small.
Then given $\epsilon > 0$, there exists 
a constant $\delta > 0$ such that if 
\begin{align}
\Vert \tilde{g}_{\infty} - g_{\infty} \Vert_{C^{2}(\partial M)} < \delta,
\end{align}
then there is no AHE metric $g_E$ near $\tilde{g}_t$ satisfying
\begin{align}
\Vert  \nabla^{\ell} \big( \tilde{x}^{- \delta_{\infty}}( \tilde{g}_t - g_E) \big) \Vert_{C^0(g_t)} < \epsilon w^{- \delta_0- \ell}
\end{align} 
for $0 \leq \ell \leq 3$, and sufficiently small gluing parameter $t$.
\end{theorem}

Using this result, we have the following ``non-example''.
\begin{nexpl}{\rm
Let $(\mathbb{D}^{2n}, g_H)$ denote the hyperbolic space of real dimension $2n$. 
This is an AHE metric, with conformal infinity the round sphere $S^{2n-1}$. 
The group $\Gamma_n$ acts on $\mathbb{D}^{2n}$, and 
the hyperbolic metric descends to the quotient.
Since $g_H$ has constant curvature equal to $-1$, the obstruction \eqref{obst3}
does not vanish at the singular point. 
Therefore, there exists no AHE metric near the 
approximate metric $\tilde{g}_t$, 
with boundary conformal class near the round metric. }
\end{nexpl}

The obstruction can be interpreted as defining a ``wall'' in the space
of boundary conformal classes. 
That is, in the space of conformal infinities which have 
orbifold Einstein fillings, the ones which have the obstruction \eqref{obst2}
vanishing at the orbifold point should generically define a hypersurface. 
It is expected that only the conformal classes on
one side of this ``wall'' have smooth fillings. This has been 
completely analyzed by Biquard in the $4$-dimensional case \cite{Biquard2}. 
However, due to space limitations, 
we will not go into more detail about the wall-crossing 
in the higher-dimensional case in this paper. 

We end the introduction with an outline of the paper. 
In Section~\ref{Calabi} we will recall some basic details about the Calabi 
metric, then we will show that the space of 
decaying infinitesimal Einstein deformations is one-dimensional. 
In Section~\ref{linearized equation} we will identify the obstruction for finding a
solution of the gauged linearized Einstein equation
with prescribed quadratic leading term,
and also give a useful reformulation of the obstruction. 
In Section~\ref{approximate} we will define the ``refined'' approximate metric 
which is an improvement of the ``na\"ive'' approximate metric which 
was defined above. Then in Section~\ref{solutionsec} we will present the 
main Lyapunov-Schmidt reduction which yields a solution of the 
Einstein equation modulo the space of obstructions, which we 
will use to prove Theorem~\ref{theoremone}. 

 The remainder of the paper will then focus on the asymptotically hyperbolic Einstein
case. In Section~\ref{AHEsec}, we will give some background on the theory of 
AHE Einstein metrics, and in particular compute the indicial roots at 
infinity for several important operators. In Section~\ref{sectionk} we will 
define a mapping from germs to boundary tensors, which considers 
solutions which blow-up at $p_0$ at a certain rate with prescribed
leading term. This will be used in Section~\ref{sectionl}, together with a 
duality argument, to prove that there is a solution of the 
linearized equation with a certain prescribed quadratic leading term at $p_0$,
and which is bounded at infinity (in the interior metric). 
This solution will then be used in Section~\ref{completionsec} to prove the 
main existence theorem, Theorem~\ref{theorem2}. Finally, in Section~\ref{completionsec},
the non-existence result Theorem~\ref{theorem3} will be proved.  

\subsection{Acknowledgments}
The authors would like to express thanks to Olivier Biquard for many discussions about his work on desingularizing Einstein orbifolds in dimension four. Also, the authors are very grateful to Rafe Mazzeo for many crucial discussions about weighted spaces and Fredholm theory of AHE metrics. We would also like to thank the anonymous referee for numerous helpful comments which greatly improved the exposition of the paper. 
\section{Infinitesimal Einstein deformations of the Calabi metric}
\label{Calabi}
We begin with a brief description of the Calabi metric. 
\subsection{The Calabi metric}
The cone $\C^n / \Gamma_n$ admits a crepant resolution $X$ which is the 
total space of the bundle $\mathcal{O}(-n) \rightarrow \mathbb{P}^{n-1}$. 
In \cite{Calabi}, Calabi proved the following.
\begin{thm}The space
$X$ admits a ${\rm{U}}(n)$-invariant Ricci-flat K\"{a}hler ALE metric denoted by $g_{cal}$.
\end{thm}
\begin{proof}
This metric is well-known, so we just give here a formula for the Calabi metric in coordinates following \cite{Kuhnel, Salamon}.  
To do this, we need to define some 1-forms and vector fields which will also be used often later in the paper. Consider the following $1$-forms
\begin{align}
	dr,J(dr),
\end{align}
where $dr$ is the radial form and $J$ is the standard complex structure. We denote the dual vectors (with respect to the Euclidean metric) by
\begin{align}
	\p_{r},J(\p_{r}).
\end{align}
Let $\theta=-J(\frac{dr}{r})$ and $\p_{\theta}$ denote the dual of  $\theta$ (with respect to the Euclidean metric). It follows that 
\begin{align}
	J(r\p_{r})=\p_{\theta}.
\end{align}
The Calabi metric in radial 
coordinates is given by\footnote{There is a typo in \cite{Kuhnel} for $g_{cal}$ in radial form. } 
\begin{align}
g_{cal} = r^{2(n-1)}(1+r^{2n})^{\frac{1-n}{n}}dr^{2}+(1+r^{2n})^{\frac{1}{n}}g_{FS}+r^{2n}(1+r^{2n})^{\frac{1-n}{n}}\theta^{2},
\end{align}
where $g_{FS}$ denotes the Fubini-Study metric on $\mathbb{P}^{n-1}$. 
There is an apparent singularity at $r=0$, which can be resolved by replacing the 
origin with the $\mathbb{P}^{n-1}$ as the zero section of $\mathcal{O}(-n) 
\longrightarrow \mathbb{P}^{n-1}$, and the metric then extends smoothly to the resolution.
\end{proof}
We will next set some notation which will be used below. Let
$\omega$ denote the Fubini-Study $2$-form on ${\mathbb{P}}^{n-1}$.
Recall that on ${\mathbb{P}}^{n-1}$ we have
\begin{align}
\label{harmonicforms}
&d\theta=2\omega, \ d\omega=0,
\end{align}
where $\theta$ is viewed as a connection form for the Hopf fibration 
$S^{2n-1} \rightarrow \mathbb{P}^{n-1}$. Next, we define 
\begin{align}
\begin{split}
\label{norms}
A(r)=\frac{r^{n-1}}{(1+r^{2n})^{\frac{n-1}{2n}}}, \ B(r)=(1+r^{2n})^{\frac{1}{n}}
, \ C(r)=(1+r^{2n})^{\frac{1-n}{2n}}r^{n},
\end{split}
\end{align}
so that 
\begin{align}
&\Norm{Adr}_{g_{cal}}=\Norm{C\theta}_{g_{cal}}=1, \ \Norm{B\omega}_{g_{cal}}^2=n-1.
\end{align}
We also define
\begin{align} 
d \overbar{r} \equiv A dr, \ 
\overbar{\theta}  \equiv C \theta, \ 
\overbar\omega  \equiv B \omega, \ 
\overbar{g_{FS}}  \equiv B g_{FS},
\end{align}
so that the Calabi metric may be written as 
\begin{align}
g_{cal}= d\overbar{r}\otimes d\overbar{r} + \overbar{g_{FS}} + \overbar{\theta}\otimes\overbar{\theta}.
\end{align}

\subsection{Infinitesimal Einstein deformations}
\label{deformation}
In this section we prove that the space of decaying infinitesimal Einstein deformations of $(X,g_{cal})$ is one-dimensional and we explicitly identify a generator. 
We first define the operators that we need. Let $(M^{m},g)$ be a Riemannian manifold
of real dimension $m$. For a $1$-form~$\omega$, 
\begin{center}
\begin{tabular}{ll}
$\delta_g : \Omega^1(M) \rightarrow \Omega^0(M)$ & $\delta_g (\omega) = - g^{ij}\nabla_i \omega_j$\\
$\delta^{\ast}_g : \Omega^1(M) \rightarrow S^2(T^*M)$ &
$\delta^{\ast}_{g}(\omega)_{ij}=\frac{1}{2}(\nabla_{i}\omega_{j}+\nabla_{j}\omega_{i})$ \\
$\mathcal{K}_g : \Omega^1(M) \rightarrow S^2_0(T^*M)$ &
$\mathcal{K}_g (\omega) = \delta^{\ast}_g \omega + \frac{1}{m} \delta_g (\omega) g$.\\
\end{tabular} 
\end{center}
For a symmetric $2$-tensor $h$, 
\begin{center}
\begin{tabular}{ll}
$\delta_g : S^2(T^*M) \rightarrow \Omega^1(M)$ & $\big{(}\delta_{g}(h)\big{)}_{i}=-g^{jk}\nabla_{j}h_{ki}$\\
$B_g : S^2(T^*M) \rightarrow \Omega^1(M)$ &
$B_{g}(h)=\delta_{g}(h)+\frac{1}{2}d (tr_{g} h)$\\
$P_g : S^2(T^*M) \rightarrow S^2(T^*M)$ &
$P_{g}(h)=\frac{1}{2}\nabla^{\ast} \nabla h -\ring{R}(h)$,\\
\end{tabular}
\end{center}
where $\nabla^{*}$ is the formal $L^2$-adjoint of $\nabla$, 
and $\ring{R}$ is the action of curvature tensor on symmetric two tensors
defined by 
\begin{align}
\ring{R}(h)_{ij} = g^{kp} g^{lq} R_{ikjl} h_{pq}.
\end{align}
We note that the operator $\nabla^{\ast} \nabla = - \Delta $, is 
the negative rough Laplacian of $h$. For simplicity of notation we will often omit the metric $g$ when it is clear from context.
We also recall the following definition from \cite{Besse}.
\begin{deff}
\label{deform}
{\em
An {\textit{infinitesimal Einstein deformation}} of an Einstein metric $g$ is a symmetric $2$-tensor field $h$ such that
\begin{align}
\delta_{g}h=0, \ P_g(h)=0, \ tr_{g}(h)=0.
\end{align}	
}
\end{deff}
We say that a tensor $h$ is {\textit{decaying}} if $h = O(r^{-\epsilon})$ 
as $r \rightarrow \infty$ for some $\epsilon >0$.  
\begin{theorem}
\label{EDCthm}
The space of decaying infinitesimal Einstein deformations of the Calabi 
metric is one-dimensional, and is spanned by the element 
\begin{align}
\label{odef}
	o=\frac{1}{1+r^{2n}}\Big(-\overbar{g_{FS}}+(n-1)d\overbar{r}\otimes d\overbar{r}+(n-1)\overbar{\theta}\otimes\overbar{\theta}\Big).
\end{align}
\end{theorem}
To prove this,  we will first characterize the decaying harmonic $(1,1)$-forms on the Calabi metric, which will then be used to characterize the 
decaying infinitesimal Einstein deformations. 
Recall that $n$ denotes the complex dimension, while $m = 2n$ denotes the real dimension. 
\begin{prop}
\label{11c}
On $(X, g_{cal})$, the space of decaying harmonic $(1,1)$-forms is one-dimensional 
and is spanned by the element
\begin{align}
\label{omegadef}
\Omega=\frac{1-n}{1+r^{2n}}d\overbar{r}\wedge\overbar{\theta}+\frac{1}{1+r^{2n}}\overbar{\omega}.
\end{align}
\end{prop}
\begin{proof} 
Define
\begin{align}
\mathcal{H}^{p,q} = \{ \eta \in C^{\infty}(\Lambda^{p,q} X) \ | \
 \eta = O(r^{1-2n}), d \eta = 0, \delta \eta = 0 \},
\end{align}
and let $H^{p,q}(X)$ denote the image of $\mathcal{H}^{p,q}(X)$ in 
$H^{p+q}(X, \mathbb{C})$ under the natural mapping $\eta \mapsto [\eta]$.
By \cite[Theorem 8.4.2]{Joyce}, the decomposition
\begin{align}
H^2(X, \mathbb{C}) = H^{2,0} \oplus H^{1,1} \oplus H^{0,2}
\end{align} 
holds, which implies that 
\begin{align}
\dim( H^{1,1}) = 1.
\end{align}
By a standard argument (see for example \cite[Section~3]{HanViaclovsky}), 
if $\eta \in C^{\infty} ( \Lambda^{p,q}X)$ satisfies 
$\eta = O(r^{-\epsilon})$ for some $\epsilon > 0$, then $\eta = O(r^{-2n+1})$
as $r \rightarrow \infty$. 

The remainder of the proof is therefore to show that the form in \eqref{omegadef} is 
both closed and co-closed, since it is obviously a $(1,1)$-form. 
Using \eqref{harmonicforms} it is easy to see that $\Omega$ is closed. To prove co-closedness, we recall that for a K\"ahler manifold $(M^{n},g_{0},J_{0},\omega_{0})$ ($n$ is complex dimension) we have the following ( see \cite{KM})
\begin{align}
\label{kal}
dV_{g_0} = \frac{\omega_{0}^{n}}{n!}.
\end{align}
Using \eqref{kal} it follows that
\begin{align}
\label{Hodgekahler}
\begin{split}
	&\ast \omega=\frac{1}{(n-2)!}B^{n-3}\omega^{n-2}\wedge d\overbar{r}\wedge \overbar{\theta}\\
& dV_g=\frac{B^{n-1}\omega^{n-1}}{(n-1)!}\wedge d\overbar{r}\wedge \overbar{\theta}.
\end{split}
\end{align}
Now straightforward calculation using \eqref{harmonicforms} and 
\eqref{Hodgekahler} shows that
\begin{align}
	\delta \Omega=-\ast d \ast \Omega=0.
\end{align}
Finally, it is easy to check that this form extends across the zero section
to define a smooth form on $X$. 
\end{proof}

Next, we complete the proof of Theorem \ref{EDCthm}.
If $h$ is an infinitesimal Einstein deformation, then 
write 
\begin{align}
h = h_+ + h_-, 
\end{align}
where $h_+$ is the $J$-invariant part, and $h_{-}$ is the $J$-anti-invariant part. 
It is proved in \cite{Besse} that 
\begin{align}
\Delta \Omega = 0, \\
\Delta I = 0, 
\end{align}
where $\Omega$ is the $(1,1)$-form associated to $h_+$,
and $I$ is the section of $\Lambda^{0,1} \otimes T^{1,0}$ 
associated to $h_-$. Since the canonical bundle is 
trivial, the conjugate Hodge star operator applied to $I$ is a
harmonic $(1,n-1)$-form. By \cite[Theorem 8.4.2]{Joyce}, this implies 
that $h_- \equiv 0$. 

 By Proposition \ref{11c}, the space of decaying infinitesimal Einstein 
deformations is spanned by the symmetric $2$-tensor
 \begin{align}
 	o(.,.)=\Omega(I(.),.).
 \end{align}
Noting that $I(\overbar{\p_{r}})=\overbar{\p_{\theta}}$ , where $\overbar{\p_{\theta}},\overbar{\p_{r}}$ are the duals of $\overbar{\theta},d\overbar{r}$ with respect to the Calabi 
metric, we obtain
\begin{align}
o(\overbar{\p_{\theta}},\overbar{\p_{\theta}})
= o(\overbar{\p_{r}},\overbar{\p_{r}}) =\frac{n-1}{1+r^{2n}}.
\end{align}
Also recall that $\omega$ is the Fubini-Study form, and that 
$I$ agrees with standard complex structure on ${\mathbb{P}}^{n-1}$, so
\begin{align}
	\omega(.,.)=g_{FS}(I(.),.)\implies \omega(I(.),.)=-g_{FS}(.,.)\implies \overbar{\omega}(I(.),.)=-\overbar{g_{FS}}(.,.)
\end{align}
(for $\overbar{\omega}$ we have $\inn{\overbar{\omega},\overbar{\omega}}_{\text{cal}}=n-1$) and $o$ becomes
\begin{align}
	o=\frac{1}{1+r^{2n}}\Big(-\overbar{g}_{FS}+(n-1)d\overbar{r}\otimes d\overbar{r}+(n-1)\overbar{\theta}\otimes\overbar{\theta}\Big),
\end{align}
which is clearly traceless. 
Now observe that we can rewrite $o$ as
\begin{align}
\label{oform}
o =\frac{1}{1+r^{2n}}\Big(-g_{cal} +n(d\overbar{r} \otimes d\overbar{r})+n(\overbar{\theta}\otimes\overbar{\theta})\Big).
\end{align}
Also since $\Omega$ is divergence-free ($\Omega$ is co-closed) and $I$ is parallel, it follows that $o$ is divergence-free, so $o$ is both traceless and divergence-free, as required. This completes the proof of Theorem \ref{EDCthm}.   

We end this section with a slightly modified statement of Theorem \ref{EDCthm}. 
\begin{theorem}
\label{Pthm}
On $(X, g_{cal})$, the space of decaying solutions of 
\begin{align}
P_g h &= 0\\
tr_g h & =0,
\end{align} 
is one-dimensional, and is spanned by the element $o$. 
\end{theorem}
\begin{proof}
The only thing we need to prove is that a solution of $P_{cal} h = 0$
which is decaying must also be divergence-free. 
To see this, since Calabi metric is Ricci-flat we get the following for linearized operator
\begin{align}
	d_{{\rm cal}} Ric(h)=\frac{1}{2}\nabla^{\ast}_{{\rm cal}}\nabla_{{\rm cal}}(h)-\ring{R}(h)-\delta^{\ast}_{{\rm cal}}B_{{\rm cal}}(h).
\end{align}
Applying the divergence operator to the equation $P_{cal} h = 0 $
yields the equation
\begin{align}
P_1 \delta_g h = 0, 
\end{align}
where $P_1 \equiv B \delta^*$. 
Since $g_{cal}$ is Ricci-flat, this implies that 
\begin{align}
\Delta_H (\delta_gh) = 0 , 
\end{align}
where $\Delta_H$ is the Hodge Laplacian. 
By \cite[Theorem 8.4.1]{Joyce}, there are no decaying harmonic 
$1$-forms on $X$, which implies that $\delta_g h = 0$. 
\end{proof}
\section{Linearized equation on the Calabi metric}
\label{linearized equation}
In this section we apply a technique in \cite{Biquard} to find the obstruction for solving linearized equation with prescribed leading quadratic term. Let 
\begin{align}
\label{quadratictensor}
	H=H_{ijkl}x^{i}x^{j}dx^{k}\otimes dx^{l}
\end{align}
be a given quadratic tensor defined on $\R^{2n}$ which is symmetric in the first two and last two indices. 
We need the following basic lemma.
\begin{lem}
\label{elliptic}
Let $v$ be a symmetric $2$-tensor with $\Vert v \Vert_{g_{cal}}=O(r^{\delta})$ for some $\delta \in (2-2n,0)$. Then there exists $u$ with $\Vert u \Vert_{g_{cal}}
=O(r^{2+\delta})$ and $P_{g_{cal}}(u)=v$ if and only if
\begin{align}
	\inn{v,o}_{L^2(g_{cal})}=0.
\end{align}
\end{lem}
\begin{proof}
This follows from standard theory of elliptic operators on weighted spaces (see \cite{bartnik} for example).
\end{proof}
Next, we have the following proposition. 
\begin{prop}
\label{limitobs}
Given ($H,\Lambda)$ with $H$ a $\Gamma_n$-invariant symmetric $2$-tensor on $\R^{2n}$ 
satisfying
 \begin{align}
 -\frac{1}{2}\Delta_{euc}H&=\Lambda g_{euc}\\
 B_{euc}H&=0,
 \end{align}
 there exists a solution $(h,\lambda)$ to 
\begin{align}
\label{oneone}
\begin{split}
P_{g_{cal}}(h)&=\Lambda {g_{cal}} + \lambda o \\
h&=H+O(r^{-2n+2+\epsilon})\\
B_{g_{cal}}(h)&=0,
\end{split}
\end{align}
if and only if
\begin{align}
\label{obstr}
\lambda = -\frac{1}{ \Vert o \Vert^{2}_{L^2}} \lim_{r\to\infty} \int_{S_{r}/\Gamma_{n}}\frac{n+1}{r}\inn{H,o}dS_{{S_{r}}/\Gamma_{n}}.
\end{align}
\end{prop}
\begin{proof}
Let $(h,\lambda)$ be a solution to (\ref{oneone}). Let $R$ be a fixed positive real number and consider the following cut-off function
\begin{align}
\label{chico}
\chi(x) = \left\{
        \begin{array}{ll}
            1 & \quad\abs{x}>2R \\
            0 & \quad \abs{x}<R
        \end{array}
    \right.
.
\end{align}
Let $h' =h-\chi H$, then $h'=O(r^{-2n+2+\epsilon})$. We have
\begin{align}
	P_{g_{cal}}(h')=-P_{g_{cal}}(\chi H)+\Lambda g_{cal}+\lambda o.
\end{align}
Next, we claim that $\Lambda{g_{cal}}-P_{g_{cal}}(\chi H)=O(r^{-2n})$. To see this, 
\begin{align}
P_{g_{cal}}(\chi H)=-\frac{1}{2}\Delta(\chi H)-\ring{R}(H)=\Lambda g_{euc}+O(r^{-2n}),
\end{align}
where we used that $H_{iikl}=\Lambda \delta_{kl}$. 
Using the fact that $o$ is traceless,
\begin{align}
	\inn{P_{g_{cal}}(\chi H)-\Lambda {g_{cal}},o}=\inn{P_{g_{cal}}(\chi H),o}. 
\end{align}
Integrating by parts yields
\begin{align}
\int_{r\le r_{0}}\inn{P_{g_{cal}}(\chi H),o} dV_{g_{cal}}=\frac{1}{2}\int_{r=r_{0}}(\inn{\chi H,\nabla_{\vec{n}}o}-\inn{\nabla_{\vec{n}}(\chi H),o}) dS_{S_{r_{0}/\Gamma_{n}}}.
\end{align}
Notice that we have the following asymptotic relations
\begin{align}
&\nabla_{\vec{n}}o = \frac{-2n}{r} o + O ( r^{-2n-2})\\
&\nabla_{\vec{n}} (\chi H) = \frac{2}{r}H + O(1),
\end{align}
as $r \rightarrow \infty$, so taking the limit as $r_0 \to \infty$ we obtain
\begin{align}
\int_X \inn{P_{g_{cal}}(\chi H),o} dV_{g_{cal}}
=- \lim_{r\to\infty} \int_{S_{r}/\Gamma_{n}}\frac{n+1}{r}\inn{H,o}dS_{{S_{r}}/\Gamma_{n}}.
\end{align}
Now applying Lemma \ref{elliptic} we should have
\begin{align}
\label{ttt}
	\inn{-P_{g_{cal}}(\chi H)+\Lambda {g_{cal}}+\lambda o,o}=0 
\end{align}
which implies
\begin{align}
\lambda = -\frac{1}{ \Vert o \Vert^{2}_{L^2}} \lim_{r\to\infty} \int_{S_{r}/\Gamma_{n}}\frac{n+1}{r}\inn{H,o}dS_{{S_{r}}/\Gamma_{n}}.
\end{align}
Conversely, assume that \eqref{obstr} is satisfied.  This implies that \eqref{ttt} is satisfied and by Lemma~\ref{elliptic} there exist a solution $h^{\prime}$ for the following problem:
\begin{align}
-P_{g_{cal}}(\chi H)+\Lambda {g_{cal}}+\lambda o=P(h^{\prime}).
\end{align}
So $h=h^{\prime}+\chi H$ is the desired solution. Next, we show that $h$ is in Bianchi gauge. Since the Calabi metric is Ricci-flat, we have
\begin{align}
	0=B_{g_{cal}}P_{g_{cal}}h=\frac{1}{2}\nabla^{\ast}\nabla B_{g_{cal}}h.
\end{align}
Also, 
\begin{align}
	h=H+O(r^{-2n+2+\epsilon})\implies B_{g_{cal}} h=O(r^{-2n+1+\epsilon}).
\end{align}
Let $\omega=B_{g_{cal}}h$. Then integrating by parts implies
\begin{align}
	0=\int_{B_{r}/\Gamma_{n}}\inn{\omega,\Delta\omega}dV_{g_{cal}}=\int_{B_{r}/\Gamma_{n}}\abs{\nabla\omega}^{2}dV_{g_{cal}}-\int_{S_{r}/\Gamma_{n}}\inn{\omega,\nabla_{\vec{n}}\omega} dS,
\end{align}
and taking the limit as $r\to \infty$ shows that $\nabla\omega=0$ and the latter implies that $\omega=0$.
\end{proof}
\subsection{Reformulation of the obstruction}
\label{reformulation}

We begin with some integral identities.
\begin{lem}
\label{integralidentity}
On $\R^{2n}$ with coordinates $(x^{1},...,x^{2n})$ we have the following:
\begin{align}
\label{sb1} \int_{S^{2n-1}}x^{i}x^{j}dS&=\frac{\omega_{2n-1}}{2n}\delta_{ij}\\
\label{sb2}\int_{S^{2n-1}}x^{i}x^{j}x^{k}x^{l}dS&=\frac{\omega_{2n-1}}{2n(2n+2)}(\delta_{kl}\delta_{ij}+\delta_{ki}\delta_{lj}+\delta_{kj}\delta_{il}).
\end{align}
\end{lem}
\begin{proof}
This is proved, for example, in \cite{Brendle}.
\end{proof}

Let $g$ be a Riemannian metric on $\R^{2n}$. The Taylor expansion of the metric in normal coordinates has the form
\begin{align}
	g_{kl}(z)=\delta_{kl}- \frac{1}{3} R_{ikjl}z^{i}z^{j}+O(\abs{z}^{3})
\end{align}
as $\abs{z}\to0$.  The quadratic term is not in Bianchi gauge, so 
we use the following lemma. 
\begin{lem} 
\label{coordlem}
Assume that $g_0$ is Einstein of dimension $2n$, 
with $Ric(g_0) = \Lambda g_0$. 
Then there exists a coordinate system near the origin such that 
\begin{align}
g_{kl}(z) = \delta_{kl}(0)+ H_{ijkl} z^i z^j +O(\abs{z}^{3}),
\end{align}
as $\abs{x}\to0$, where 
\begin{align}
\label{Hform}
H_{ijkl} z^i z^j = - \frac{1}{3} R_{ikjl}z^{i}z^{j}
- \frac{\Lambda}{3(n + 1)} ( \delta_{kl} |z|^2 + 2 z^k z^l).
\end{align}
\end{lem}
\begin{proof}
It is easy to verify that 
\begin{align}
B_{euc} (H_{ijkl} z^i z^j) = 0.
\end{align}
Finally, we note that  
\begin{align}
\delta_{euc}^* ( |z|^2 z^i) = \delta_{ij} |z|^2 + 2 z^i z^j,
\end{align} 
so the added term is a Lie derivative of a cubic vector field. 
A standard argument, see for example \cite[Section~7]{AcheViaclovsky},
implies that there is a coordinate system with the stated property. 
\end{proof}

\begin{prop}
Let $H=H_{ijkl}x^{i}x^{j}dx^{k}dx^{l}$ be the symmetric $2$-tensor with 
\begin{align}
H_{ijkl} = - \frac{1}{3} R_{ikjl} - \frac{\Lambda}{3(n+1)} 
(\delta_{ij} \delta_{kl} + 2 \delta_{ik} \delta_{jl} ).
\end{align}
Then the constant $\lambda$ in Proposition \ref{limitobs} is given by 
\begin{align}
\label{obst}
\lambda = \frac{1}{\Vert o \Vert_{L^2}} \frac{\omega_{2n-1}}{n}\Big(  \frac{2-n}{2} \Lambda - \frac{1}{8}  \inn{\mathcal{R}(\omega),\omega} \Big).
\end{align}
\end{prop}
\begin{proof}
It follows from \eqref{oform} that in ALE coordinates, we have 
\begin{align}
o=\frac{1}{r^{2n}}\Big( -g_{euc} +n(d r \otimes d r)+nr^2 ( \theta \otimes \theta) \Big) + O(r^{-2n-1})
\end{align}
as $r \rightarrow \infty$. Defining 
\begin{align}
o_{euc} = \frac{1}{r^{2n}}\Big(-g_{euc} +n(d r \otimes d r)+nr^2( \theta \otimes \theta)\Big),
\end{align}
we have
\begin{align}
\lim_{r\to\infty} \int_{S_{r}/\Gamma_{n}}\frac{n+1}{r}\inn{H,o} dS_{{S_{r}}/\Gamma_{n}}
= \int_{\abs{x} = 1} (n+1)\inn{H,o_{euc}} d S_{{S_{r}}/\Gamma_{n}}.
\end{align}
The inner product is  
\begin{align}
\label{ipe}
	\inn{H,o_{euc}}=-\inn{H,g_{euc}}+n\inn{H,d r\otimes d r}+n\inn{H,r^2 \theta 
\otimes \theta}.
\end{align}
We lift up to $\mathbb{R}^{2n}$, and next we calculate each term in \eqref{ipe}; 
the first term is
\begin{align}
	\inn{H,g_{euc}}=H_{ijll}x^{i}x^{j},
\end{align}
so by \eqref{sb1}, 
\begin{align}
\label{com1}
\int_{|x| = 1}\inn{H,g_{euc}} dS = \frac{\omega_{2n-1}}{2n} H_{iikk}.
\end{align}
Also, since $r^{2}=(x^{i})^{2}$, we have
\begin{align}
rdr&=x^{i}dx^{i} \\
r^{2}dr\otimes dr&=x^{i}x^{j}dx^{i}dx^{j}.
\end{align}
Letting $J$ denote the standard complex structure on $\R^{2n}$, then $\theta=-J(\frac{dr}{r})$, so we have
\begin{align}
	(r\theta) \otimes (r\theta)=x^{p}x^{q}J_{k}^{p}J_{l}^{q}dx^{k}dx^{l}.
\end{align}
Clearly,
\begin{align}
r^{2}\inn{\theta\otimes \theta,H}&=x^{i}x^{j}x^{p}x^{q}J^{p}_{k}J^{q}_{l}H_{ijkl},\\
r^{2}\inn{dr\otimes dr,  H}&=x^{i}x^{j}x^{k}x^{l}H_{ijkl}.
\end{align}
Using \eqref{sb2}, this implies that 
\begin{align}
\label{com2}
\int_{\abs{x}=1} \inn{ H, d r\otimes d r} dS
= \frac{\omega_{2n-1}}{2n(2n+2)} ( H_{iikk} + H_{ikik} + H_{ikki}).
\end{align}
Also, \eqref{sb2} implies 
\begin{align}
\label{com3}
\int_{\abs{x}=1} \inn{ H, \theta \otimes \theta } dS
=\frac{\omega_{2n-1}}{2n(2n+2)}(J^{p}_{k}J^{p}_{l}H_{iikl}+J^{p}_{k}J^{q}_{l}H_{pqkl}+J^{p}_{k}J^{q}_{l}H_{qpkl}).
\end{align}
Putting together \eqref{com1}, \eqref{com2}, and \eqref{com3}, we obtain
\begin{align}
\begin{split}
\label{com4}
\int_{\abs{x}=1}(n+1)\inn{H,o_{euc}}dS =
-&\frac{\omega_{2n-1}}{2n}(n+1)H_{kkll}
+\frac{\omega_{2n-1}}{4}(H_{iikk} +  H_{ikik} + H_{ikki} )\\
+ &\frac{\omega_{2n-1}}{4} ( J^{p}_{k}J^{p}_{l}H_{iikl}+J^{p}_{k}J^{q}_{l}H_{pqkl}+J^{p}_{k}J^{q}_{l}H_{qpkl}).
\end{split}
\end{align}
Next we compute the terms involving $H$ appearing in \eqref{com4}.
First, 
\begin{align}
\label{tm1}
H_{kkll} &= - \frac{1}{3} R_{klkl} -  \frac{\Lambda}{3(n+1)} (\delta_{kk} \delta_{ll} 
+ 2 \delta_{kl} \delta_{kl})
= - 2n \Lambda.
\end{align}
The next two terms are
\begin{align}
\label{tm2}
H_{ikik} &= - \frac{1}{3} R_{iikk} -  \frac{\Lambda}{3(n+1)} ( \delta_{ik} \delta_{ik}
+ 2 \delta_{ii} \delta_{kk} ) = -  \frac{\Lambda}{3(n+1)} ( 2n  + 8 n^2 ),
\end{align}
and
\begin{align}
\label{tm3}
H_{ikki} &= - \frac{1}{3} R_{ikki} -  \frac{\Lambda}{3(n+1)} ( \delta_{ik} \delta_{ik}
+ 2 \delta_{ik} \delta_{ik} ) 
= \frac{2n \Lambda}{3} - \frac{ 2n \Lambda}{n+1}.
\end{align}
The first term involving $J$ is  
\begin{align}
\label{tm4}
J^{p}_{k}J^{p}_{l}H_{iikl} 
=  - 2n \Lambda.
\end{align}
The next term is
\begin{align}
\label{tm5}
J^{p}_{k}J^{q}_{l}H_{pqkl} 
=- \frac{1}{3} \inn{\mathcal{R}(\omega),\omega}
-  \frac{2n \Lambda}{3(n+1)}.
\end{align}
For the last term, we have
\begin{align}
\label{com35}
J^{p}_{k}J^{q}_{l}H_{qpkl} 
=  - \frac{1}{3} J^{p}_{k}J^{q}_{l}R_{qkpl}  
+  \frac{2n \Lambda}{3(n+1)}.
\end{align}
Note that using the Bianchi identity, we have
\begin{align}
\label{com5}
0 
=  J^{p}_{k}J^{q}_{l} R_{qkpl} -  J^{k}_{p}J^{q}_{l} R_{kpql} +  J^p_k J^l_q R_{pqlk}.
\end{align}
The first and third terms in the last line of \eqref{com5} are easily seen to be 
equal, thus we have 
\begin{align}
\label{com6}
 J^{p}_{k}J^{q}_{l} R_{qkpl} = \frac{1}{2}  \inn{\mathcal{R}(\omega),\omega}.
\end{align}
Plugging \eqref{com6} into \eqref{com35}, we obtain 
\begin{align}
\label{tm6}
J^{p}_{k}J^{q}_{l}H_{qpkl} & =  - \frac{1}{6}   \inn{\mathcal{R}(\omega),\omega} 
+  \frac{2n \Lambda}{3(n+1)}.
\end{align}
We now insert \eqref{tm1}, \eqref{tm2}, \eqref{tm3}, \eqref{tm4}, 
\eqref{tm5}, and \eqref{tm6} into \eqref{com4} to obtain 
\begin{align}
\begin{split}
\label{com8}
\int_{\abs{x}=1}(n+1)\inn{H,o}dS =
\omega_{2n-1}\Big(  \frac{2-n}{2} \Lambda - \frac{1}{8}  \inn{\mathcal{R}(\omega),\omega} \Big).
\end{split}
\end{align}
Finally, since we are integrating over $S^{2n-1}/ \Gamma$, where $\Gamma$ is of order $n$, the area of the unit sphere is $\frac{ \omega_{2n-1}}{n}$, which accounts for the extra factor of $\frac{1}{n}$ appearing in~\eqref{obst}. 
\end{proof}
\section{The refined approximate metric}
\label{approximate}
The approximate metric described in the introduction will be referred to as
the ``na\"ive'' approximate metric. In this section, we will construct a ``refined'' approximate metric, which will be necessary to compute the obstruction. 
The construction closely follows that in \cite{Biquard} with minor modifications. 
Recall that the Einstein orbifold is denoted by $(M_{0},g_{0})$ and
$(X,g_{cal})$ denotes the Calabi metric. Also, recall that $\{x^i\}, i = 1 \dots 2n,$ are ALE coordinates on $X$, and that $g_0$ satisfies
\begin{align}
     Ric(g_{0})=\Lambda g_{0}.
\end{align}
Near $p_0$, we use the coordinate system $\{z^i\}, i = 1 \dots 2n$,
from Lemma \ref{coordlem}, where $g_0$ takes the form 
\begin{align}
g_{0}= g_{euc}+H + o(|z|^2), 
\end{align}
with
\begin{align}
H=H_{ijkl}z^{i}z^{j}dz^{k}dz^{l},	
\end{align}
given by \eqref{Hform}. 

On the Calabi metric, using Proposition \ref{limitobs}, 
there exists a solution $(h,\lambda)$ for the following problem
\begin{align}
\label{eeee}
\begin{split}
P_{g_{cal}}Ric (h)&=\Lambda g_{cal}+\lambda o, \\
h&=H+O(r^{2-2n+\epsilon}),\\
B_{g_{cal}}(h) &= 0,
\end{split}
\end{align}	
for any $\epsilon > 0$, where $\lambda$ is given by \eqref{obst},
as $r \rightarrow \infty$. 
Define
\begin{align}
	h_{t}={g_{cal}} + th.
\end{align}
Consider the Calabi metric on $\C^{n}-\{0\}/\Gamma_{n}$ and define the following two regions
\begin{align}
&DZ^t \equiv \{ \frac{1}{2}t^{-\frac{1}{4}}<r<2t^{-\frac{1}{4}}\}\\
&{{X}}^{t} \equiv \{ r<2t^{-\frac{1}{4}}\},
\end{align}
(DZ is short for damage zone).
Now define the following function on $\C^{n}-\{0\}/\Gamma_{n}$
\begin{align}
	\rho=\begin{cases}
r\phantom{=}&\text{when $r>2$}\\
1\phantom{=}&\text{when $r\le 1$},
\end{cases}
\end{align}

\begin{lem}
\label{estimate}
For $t$ sufficiently small, $h_{t}$ is a Riemannian metric on ${{X}}^{t}$ and the following estimate holds on this region
\begin{align}
	\abs{\nabla^{k} (Ric_{h_{t}}-t\Lambda h_{t}-t\lambda o)}_{g_{cal}}\le c_{k}t^{2}\rho^{2-k},
\end{align}	
for some constants $c_{k}$ and any $k\ge 0$.
\end{lem}
\begin{proof}
Since $h$ has quadratic leading term and $g_{cal}$ is ALE, there exist a constant $c$ so that $\abs{h}_{g_{cal}}(p)\le c\rho^{2}(p)$ for any $p\in X^{t}$. It follows that
\begin{align}
	\abs{th}_{g_{cal}}(p)\le ct\rho^{2}(p)\le 2ct^{\frac{1}{2}}.
\end{align}
The right hand side is independent of $p$  so $h_{t}$ is a Riemannian metric on ${X}^{t}$ for small $t$. Recall that
\begin{align}
	d_{g_{cal}}Ric(h)=\Lambda g_{cal}+\lambda o 
\end{align} 
so we have
\begin{align}
	&Ric(g_{cal}+th)=Ric(g_{cal})+td_{g_{cal}}Ric(h)+Q(th),
\end{align}
where $Q$ is a nonlinear term.
For any Riemannian metric $g$, and symmetric tensor $h$ small, 
recall the well-known expansion 
\begin{align}
\label{Rmexp0}
Rm_{g +h } = Rm_g + (g+h)^{-1} * \nabla^2 h + (g+h)^{-2} * \nabla h * \nabla h,
\end{align}
where $*$ denotes various tensor contractions (see, for example \cite[Section 3]{GurskyViaclovsky}). Contracting again, we obtain
\begin{align}
\label{Rmexp}
Ric_{g +h} = (g+h)^{-1}* Rm_g + (g+h)^{-2} * \nabla^2 h + (g+h)^{-3} * \nabla h * \nabla h.
\end{align}
It follows that for $t$ sufficiently small
\begin{align}
\label{Qdef}
\abs{Q(th)}\le c_{1}t^{2}( |Rm_{g_{cal}}| |h|^2 + \abs{h}\abs{\nabla^{2}h}+\abs{\nabla h}^{2}).
\end{align}
Using \eqref{eeee}, it follows that
\begin{align}
&\abs{Ric(h_{t})-t\Lambda h_{t}-\lambda o}=c_{1}t^{2}(
|Rm_{g_{cal}}| |h|^2+ \abs{h}\abs{\nabla^{2}h}+\abs{\nabla h}^{2})
\end{align}
Note that  
\begin{align}
	\abs{\nabla^{k}h}\le c_{k}\rho^{2-k},
\end{align}
so the lemma for $k=0$ follows. For $k>0$ the proof is similar, and is omitted. 
\end{proof}
Recalling the attaching map $\phi_{t}$ defined in \eqref{attachingmap}, consider
\begin{align}
	M^{t}=M_{0}^{t} \cup_{\phi_{t}} {{X}}^{t}
\end{align}
where
\begin{align}
M_0^t &= M \setminus B(p_0, (1/2) t^{1/4}).
\end{align}
To define the refined approximate metric on $M^t$, let $\chi$ be as in \eqref{cutoff},
\begin{align}
\label{rmdef}
g_t =
\begin{cases}
g_{0}  & M \setminus B(p_0, 2 t^{1/4}) \\
(1 - \chi_{t}(r) )\phi_{t}^* g + \chi_{t}(r) t h_{t} & (1/2) t^{-1/4} < r < 2 t^{-1/4}\\
t h_{t}  & r < (1/2) t^{-1/4} \\
\end{cases}.
\end{align}
We have the following estimation.
\begin{lem}
\label{estimate2}
	For each integer $k\ge 0$, there exist a constant $c_{k}$ so that $g_{t}$ satisfies the following estimate on ${{X}}^{t}$
	\begin{align}
		\abs{\nabla^{k} (Ric_{g_{t}}-\Lambda g_{t}-t\lambda \chi_{t} (\rho) o)}_{g_{cal}}\le c_{k}t^{2}\rho^{2-k},
		\end{align}
		where the covariant derivatives and norm are taken with respect to $g_{cal}$.
\end{lem}
\begin{proof}
If $r< \frac{1}{2}t^{-\frac{1}{4}}$ then $g_{t}=th_{t}$ and by Lemma \ref{estimate} we are done. If 
\begin{align}
\label{dz}
	(1/2) t^{-1/4} < r < 2 t^{-1/4}
\end{align}
then $g_{t}=(1 - \chi_{t}(r) )\phi_{t}^* g_{0} + \chi_{t}(r) t h_{t}$. Recall that
\begin{align}
&\phi_{t}^{\ast}g_{0}=tg_{euc}+t^{2}H+tO(t^{\frac{3}{2}}r^{3})\\
&th_{t}=tg_{euc}+t^{2}H+tO(r^{-2n})+tO(r^{2-2n+\epsilon}).
\end{align}
Since the quadratic terms agree, we have the estimate
\begin{align}
&g_{t}=tg_{euc}+t^{2}H+tO(t^{\frac{3}{2}}\rho^{3}).
\end{align}
Writing 
 \begin{align}
 \theta=g_{t}-tg_{euc},	
 \end{align}
expanding the Ricci tensor near $t\cdot g_{euc}$ we have
\begin{align}
&Ric_{g_{t}}=d_{t \cdot euc}Ric(\theta)+Q_{euc}(\theta),
\end{align}
where $Q_{euc}(\theta)$ satisfies the estimate
\begin{align}
&t^{2}\abs{Q_{euc}(\theta)}\le 
c(\abs{\theta}\abs{\nabla^{2}\theta}+\abs{\nabla \theta}^{2}),
\end{align}
for some constant $c$, where the norms are with respect to $g_{euc}$. Note that \eqref{dz} implies that
\begin{align}
t\lambda \chi_{t} (\rho) o=O(t^{2}\rho^{2}).
\end{align}
From \eqref{eeee} it follows that 
\begin{align}
	&d_{t\cdot euc}Ric(\theta)=t\Lambda g_{euc}+O(t^{2}\rho^{2}),
\end{align}
where $O$ is with respect to the Euclidean metric. Finally we obtain
\begin{align}
Ric_{g_{t}}-\Lambda g_{t}-t\lambda \chi_{t} (\rho) o=O(t^{2}\rho^{2}),
\end{align}
so the lemma for $k=0$ follows. For general $k$ the proof is similar and is omitted. \end{proof}

\section{Solution of the Einstein equation modulo obstructions}
\label{solutionsec}
\subsection{Function Spaces}
In this section we define the function spaces that we will need in order to 
apply the implicit function theorem. 
Let $E$ be a tensor bundle over $X^{t}$ and $\delta_{0}\in \R$ be a weight. 
We define $C^{k,\alpha}_{\delta_{0}}(X^{t},E)$ to be
\begin{align}
\label{n1}
	&{\abs{s}}_{\alpha}\equiv\sup_{d(x,y)<i_{g_{cal}}}\frac{\abs{s(x)-s(y)}}{d(x,y)^{\alpha}}\\
	&\Norm{s}_{C^{k,\alpha}_{\delta_{0}}}\equiv\sum_{i=0}^{k}(\sup \rho^{\delta_{0}+i}\abs{\nabla^{i}s})+{\abs{\rho^{\delta_{0}+k+\alpha}\nabla^{k} s }}_{\alpha}.
\end{align}
Note $s(x)$ and $s(y)$ in \eqref{n1} 
are compared by means of parallel transport.
Let $i_{p_{0}}$ denote the injectivity radius at the point $p_{0}$. Define the following function
\begin{align}
		\tilde{z}=\begin{cases}
|z|\phantom{=}&\text{for $|z|<\frac{1}{2}i_{p_{0}}$}\\
1\phantom{=}&\text{on $M_{0}-B(p_{0},z_{p_{0}})$}
\end{cases}.
\end{align}
We can define a similar norm $C^{k,\alpha}_{\delta_{0}}(M_{0},E)$ for tensor bundles over $M_{0}$
\begin{align}
	\Norm{s}_{C^{k,\alpha}_{\delta_{0}}}\equiv\sum_{i=0}^{k}(\sup {\tilde{z}}^{\delta_{0}+i}\abs{\nabla^{i}s})+{\abs{{\tilde{z}}^{\delta_{0}+k+\alpha}\nabla^{k} s }}_{\alpha}.
\end{align} 
When $(M_{0},\p M_{0},g_{0},x)$ is AH, for another weight $\delta_{\infty} \in \mathbb{R}$,  
we define $C^{k,\alpha}_{\delta_{0},\delta_{\infty}}(M_{0},E)$ by the norm
\begin{align}
\Norm{s}_{C^{k,\alpha}_{\delta_{0},\delta_{\infty}}}\equiv\Norm{\tilde{x}^{-\delta_{\infty}}s}_{C^{k,\alpha}_{\delta_{0}}},
\end{align}
where $\tilde{x}$ is defined in \eqref{txdef}. 
We also need to define function spaces $C^{k,\alpha}_{\delta_{0},\delta_{\infty};t}$ on 
the topological desingularization $M^{t}$. Consider the following norms
\begin{align}
\label{Normss}
\Norm{s}_{C^{k,\alpha}_{\delta_{0},\delta_{\infty};t}} &\equiv t^{\frac{\delta_{0}+l}{2}}\Norm{\chi_{t}s}_{C^{k,\alpha}_{\delta_{0}}(X^{t})}+\Norm{(1-\chi_{t})s}_{C^{k,\alpha}_{\delta_{0},\delta_{\infty}}(M_{0}^{t})},\\
\Norm{s}_{C^{k,\alpha}_{\delta_{0},\delta_{\infty};t}}' &\equiv \sum_{i=0}^{k}(\sup w^{\delta_{0}+i}\abs{\nabla^{i}( \tilde{x}^{- \delta_{\infty}}s)}+\abs{w^{\delta_{0}+k+\alpha}\nabla^{k}(  
\tilde{x}^{- \delta_{\infty}}s  )}_{\alpha},
\end{align}
where $l=r-s$ for an $(r,s)$ tensor $E$ and $w$ is the weight function and the covariant derivatives are with respect to $g_{t}$ in last expression. It is easy to see that these norms are uniformly equivalent when $t$ is sufficiently small, so we will use them interchangeably, and omit the prime notation. The exponent of $t$ in \eqref{Normss} follows because
\begin{align}
t^{\frac{l}{2}}\abs{s}_{\frac{g}{t}}=\abs{s}_{g},
\end{align}
and since $\rho$ is comparable to $\frac{ \tilde{z}}{\sqrt{t}}$ in the damage zone.

In the following, we fix $\delta_{0}>0$ and $\delta_{\infty}>0$ sufficiently small.
\subsection{Lyapunov-Schmidt reduction}
\label{modulo}
Define the space of obstructions 
\begin{align}
\mathcal{O}_{t} \subset C^{\alpha}_{\delta_0 + 2,
\delta_{\infty};t}(M^t)
\end{align}
 to be the one-dimensional span of $\chi(t)o$. Define the following operator
\begin{align}
	\Phi_{g_{t}}(g)=Ric(g)-\Lambda g +\delta_{g}^{\ast}B_{g_{t}}g.
\end{align}
The essential ingredient to the gluing argument is the following 
implicit function theorem \cite{Biquard}.
\begin{lem}
\label{implicitfunction}
Let $\Phi:E\to F$ be a differentiable map between Banach spaces and
\begin{align}
	Q=\Phi-\Phi(0)-d_{0}\Phi.
\end{align}	
Assume there exist positive constants $q,r_{0},c$ such that
\begin{align}
\label{Qesti}
&\Norm{Q(x)-Q(y)}\le q\Norm{x-y}(\Norm{x}+\Norm{y})\text{  } x,y\in B(0,r_{0})\\
	&d_{0}\Phi \text{ is an isomorphism and } \Norm{(d_{0}\Phi)^{-1}}\le c.
\end{align}
If $r\le \min{(r_{0},\frac{1}{2qc})}$ and $\Norm{\Phi(0)}\le \frac{r}{2c}$ then $\Phi(x)=0$ has a unique solution in $B(0,r)$.
\end{lem}

Let $v\in \R$ and consider the following metric on $M^{t}$
\begin{align}
	g_{t,v}=g_{t}+tv\chi_{t} o.
\end{align}
Notice that
\begin{align}
	\Norm{t\chi_{t}o}_{C^{2,\alpha}_{\delta_{0},\delta_{\infty};t}}=O(t^{\frac{\delta_{0}}{2}}).
\end{align}
\begin{prop}
\label{obs}
Assume $(M_{0},g_{0})$ is a non-degenerate Einstein orbifold with one singular point $p_{0}$ with orbifold group $\Gamma_n$ ($M_{0}$ can either be compact or AHE). Then there 
exist $\eta,\epsilon >0$ such that for any $t< \eta$ and $v \in \mathbb{R}$ with 
$\abs{v}\le \epsilon$, there exist a unique metric $\hat{g}_{t,v}$ solving the following equation
\begin{align}
\Phi ( \hat{g}_{t,v}) \in \mathcal{O}_{t},
\end{align}
where $\hat{g}_{t,v}$ satisfies the following
\begin{align}
\label{conditions1}
&\Norm{\hat{g}_{t,v}-g_{t}}_{C^{2,\alpha}_{\delta_{0},\delta_{\infty}}} \le \epsilon t^{\frac{\delta_{0}}{2}},\\
\label{conditions2}
&\int_{\mathbb{P}^{n-1}}\Big \langle \frac{1}{t}({\hat{g}}_{t,v}-g_{t,v}),o 
\Big \rangle_{cal}dV_{cal}  =v.
\end{align}
Also, if we write $\Phi( \hat{g}_{t,v}) =\lambda(t) (\chi_{t}o)$,
then we have the following estimation
\begin{align}
	&\Norm{\hat{g}_{t,v}-g_{t,v}}_{C^{2,\alpha}_{\delta_{0},\delta_{\infty};t}}\le ct^{1+\frac{\delta_{0}}{4}},\\
\label{lexp}
	&\lambda(t)=t\lambda+O(t^{\frac{3}{2}-\delta})
\end{align}
for any $\delta > 0$ as $t \rightarrow 0$. 
\end{prop}
\begin{proof}
 The following proof is very similar to \cite[Proposition 9.1]{Biquard}.
However, since scalings of various terms are different in higher dimensions, 
we give a fairly detailed outline here, omitting some computations. 
To begin, define the following operator
\begin{align}
&\Phi:t^{-\frac{\delta_{0}}{2}}\tilde{C}^{2,\alpha}_{\delta_{0},\delta_{\infty};t}\oplus \R \to t^{-\frac{\delta_{0}}{2}} C^{\alpha}_{\delta_{0}+2,\delta_{\infty};t}\\
&\Phi(h,x)=Ric_{g_{t,v,h}}-\Lambda g_{t,v,h}+\delta^{\ast}_{g_{t,v,h}}B_{g_{t}}g_{t,v,h}-x\chi_{t}o	
\end{align}
where $g_{t,v,h}=g_{t,v}+h$ and $\tilde{C}^{2,\alpha}_{\delta_{0},\delta_{\infty};t}$ is space of all $C^{2,\alpha}_{\delta_{0},\delta_{\infty};t}$ tensors like $h$ which satisfies:
\begin{align}
	\int_{\mathbb{P}^{n-1}}\Big \langle\frac{1}{t}(h-g_{t}),o\Big \rangle_{cal}dV_{cal}=v.
\end{align}
For simplicity of notation, in the remainder of the proof we will omit the variable $v$. 
We want to apply Lemma \ref{implicitfunction} for $E=t^{-\frac{\delta_{0}}{2}}\tilde{C}^{2,\alpha}_{\delta_{0},\delta_{\infty};t}\oplus \R$ and $F=t^{-\frac{\delta_{0}}{2}} C^{\alpha}_{\delta_{0}+2,\delta_{\infty};t}$. 
From Lemma \ref{estimate} it follows that
\begin{align}
{\Norm{\Phi(0,0)}}_{C^{0,\alpha}_{\delta_{0}+2,\delta_{\infty};t}}=O(t^{1+\frac{\delta_{0}}{4}}).
\end{align}
The linearization of $\Phi$ at $(0,0)$ is given by
\begin{align}
d_{(0,0)}\Phi(h,x)=P_{t,v}h-x\chi_{t}o+\frac{1}{2}(Ric_{g_{t,v}}\circ h+h \circ Ric_{g_{t,v}}-2\Lambda h).
\end{align}
Also the quadratic terms are given by
\begin{align}
Q(h,x)=&\Phi(h,x)-\Phi(0,0)-d_{(0,0)}\Phi(h,x).
\end{align}
Using the expansion \eqref{Rmexp}, and a similar expansion for
the gauge term (see \cite[Section 3]{GurskyViaclovsky}), 
it is not hard to see that there exist 
constants $r_{0}$ and $q$ so that \eqref{Qesti} is satisfied. 
We need to show that $d_{(0,0)}\Phi$ is an isomorphism with bounded inverse. 
Equivalently, we need to show that there exists a constant $c>0$ so that
\begin{align}
t^{-\frac{\delta_{0}}{2}}\Norm{h}_{C^{2,\alpha}_{\delta_{0},\delta_{\infty};t}}+\abs{x}\le ct^{-\frac{\delta_{0}}{2}}\Norm{d_{(0,0)}\Phi}_{C^{0,\alpha}_{\delta_{0}+2,\delta_{\infty};t}}.	
\end{align}
Assume this is not true, then there exist a sequence $t_{i}\to 0$ and $(h_{i},x_{i})$ such that
\begin{align}
\label{zeroest0}
t_{i}^{-\frac{\delta_{0}}{2}}\Norm{h_{i}}_{C^{2,\alpha}_{\delta_{0},\delta_{\infty};t_{i}}}+\abs{x_{i}}=1,
\end{align}
but 
\begin{align}
\label{zeroest}
t_{i}^{-\frac{\delta_{0}}{2}} \Norm{d_{(0,0)}\Phi(h_{i},x_{i})	}_{C^{0,\alpha}_{\delta_{0}+2,\delta_{\infty};t_{i}}}\to 0.
\end{align}
Let us first consider that case that $x_{i}\to 0$. Then without loss of generality we may assume that
\begin{align}
t_{i}^{-\frac{\delta_{0}}{2}}\Norm{h_{i}}_{C^{0}_{\delta_{0},\delta_{\infty};t_{i}}}\to 1,
\end{align}
and 
\begin{align}
	t_{i}^{-\frac{\delta_{0}}{2}}\Norm{d_{(0,0)}\Phi(h_{i},x_{i})}_{C^{0,\alpha}_{\delta_{0}+2,\delta_{\infty};t_{i}}}\to 0.
\end{align}
To get a contradiction, we briefly outline the standard blow-up argument. Let $p_{i}\in M_{t}$ be a sequence of points such that norm is attained. By passing to a subsequence we can assume that $p_{i}\to p$. There are three cases for $p$. 
\begin{itemize}
\item
If $p\in M_{0}$ then by standard elliptic theory we can obtain a limit $h_{\infty}$
satisfying $\Norm{h_{\infty}}_{C^{0}_{\delta_{0},\delta_{\infty}}}=1$ and $P_{g_{0}}h_{\infty}=0$. But this is impossible because we assumed that $M_{0}$ is non-degenerate.
\item
If $p\in X$ then we can obtain a limit $h_{\infty}$ with $\Norm{h_{\infty}}_{C^{0}_{\delta_{0}}}=1$ and $P_{g_{cal}}h_{\infty}=0$. From Theorem \ref{Pthm}, 
$h_{\infty} = c \cdot o$, for some constant $c$. But then \eqref{conditions2} implies that $c=0$, a contradiction.  
\item
If $p \in \mathbb{C}^{n}/\Gamma_{n}$ then we can obtain a limit $h_{\infty}$ with $\Norm{h_{\infty}}_{C^{0}_{\delta_{0}}}=1$ and $\nabla^{\ast}\nabla h_{\infty}=0$, but this is impossible since the indicial roots are the integers and $\delta_0$ was chosen small but non-zero.
\end{itemize}
Next, we show that $x_{i}\to 0$. To prove this, we need the following estimates,
\begin{align}
\label{est1}
&\abs{\inn{u,\chi_{t} o}_{L^2(g_{t})}}\le C t^{n-2-\frac{\delta_{0}}{2}}\Norm{u}_{C^{0,\alpha}_{\delta_{0}+2,\delta_{\infty};t}},\\
\label{est2}
&\abs{\inn{P_{t}h,\chi_{t} o}_{L^2(g_{t})}}\le C t^{n-\frac{3}{2}-\frac{\delta_{0}}{4}}\Norm{h}_{C^{2,\alpha}_{\delta_{0},\delta_{\infty};t}},\\
\label{est3}
&\inn{\chi_{t}o,\chi_{t}o}_{L^2(g_{t})}=t^{n-2}\big(  
\inn{o, o}_{L^2(g_{cal})} + O(t^{\frac{n}{2}}) \big),
\end{align}
for some constant $C >0$. The proofs of these estimates are almost exactly the same as those of \cite[(71)-(73)]{Biquard} and are omitted. 
Using \eqref{zeroest0}, \eqref{est2}, and \eqref{est3}, we can estimate
\begin{align}
\inn{d_{(0,0)}\Phi(h_{i},x_{i}),\chi_{t}o}_{L^2(g_{t})} = x_i t_i^{n-2} \inn{o, o}_{L^2(g_{cal})}
+ o(t_i^{n-2}),
\end{align}
as $i \rightarrow \infty$. However, using \eqref{zeroest} and \eqref{est1}, 
we also have
\begin{align}
\inn{d_{(0,0)}\Phi(h_{i},x_{i}),\chi_{t}o}_{L^2(g_{t})} = o(t_i^{n-2}),
\end{align}
as $i \rightarrow \infty$. It follows that
\begin{align}
x_i t_i^{n-2} \inn{o, o}_{L^2(g_{cal})}= o(t_i^{n-2}),
\end{align}
as $i \rightarrow \infty$. 
Consequently, we must have $x_{i}\to 0$ as $i \to \infty$.	

It remains to prove the expansion for $\lambda(t)$. From Lemma \ref{implicitfunction}, 
for each $t$ sufficiently small, 
we obtain a solution $(h_t,x_t)$ of $\Phi(h_t,x_t) = 0$. Defining
$\hat{g}_{t}=g_{t}+h_{t}$, we have $\Phi_{g_t}( \hat{g}_{t}) =\lambda(t) (\chi_{t}o)$. 
Notice that
\begin{align}
\inn{\Phi_{g_t}( g_{t}+h_t),\chi_{t} o}_{L^2(g_{t})}=\lambda(t)\inn{	\chi_{t}o,\chi_{t}o}_{L^2(g_{t})}.
\end{align}
On the other hand
\begin{align}
\Phi_{g_{t}}(\hat{g}_{t})=\Phi_{g_{t}}(g_{t})+d_{g_{t}}\Phi_{g_t}(h_{t})	+Q_{g_{t}}(h_{t}),
\end{align}
which implies that
\begin{align}
\inn{\Phi_{g_{t}}(\hat{g}_{t}),\chi_{t}o}_{L^2(g_{t})}=\inn{\Phi_{g_{t}}(g_{t}),\chi_{t}o}_{L^2(g_{t})}+\inn{d_{g_{t}}\Phi_{g_t}(h_{t}),\chi_{t}o}_{L^2(g_{t})}+\inn{Q_{g_{t}}(h_{t}),\chi_{t}o}_{L^2(g_{t})}.
\end{align}
Using \eqref{est1}-\eqref{est3}, Lemma \ref{estimate2}, and the expansion \eqref{Rmexp}, respectively, we can estimate each of the above terms by
\begin{align}
\inn{\Phi_{g_{t}}(g_{t}),\chi_{t}o}_{L^2(g_{t})} &=  
\lambda t^{n-1}\Norm{o}_{L^{2}(g_{cal})}^2 + O(t^n),\\
\inn{d_{g_{t}}\Phi(h_{t}),\chi_{t}o}_{L^2(g_{t})} &=  O( t^{n- \frac{1}{2}- \frac{\delta_{0}}{4}}) \\
\inn{Q_{g_{t}}(h_{t}),\chi_{t}o}_{L^2(g_{t})} &= O( t^{n-\frac{\delta_{0}}{2}}),
\end{align}
as $t \rightarrow 0$, 
and we obtain
\begin{align}
\lambda(t)t^{n-2}\Norm{o}_{L^{2}(g_{cal})}^2=t^{n-2}\big(\lambda t\Norm{o}_{L^{2}(g_{cal})}^2+
O(t^{\frac{3}{2}-\frac{\delta_{0}}{4}})\big),
\end{align}
which implies
\begin{align}
\lambda(t)=t\lambda+O(t^{\frac{3}{2}-\frac{\delta_{0}}{4}}), 	
\end{align}
as $t \rightarrow 0$. 
\end{proof}

\begin{proof}[Proof of Theorem \ref{theoremone}]
We first assume that $g_0$ is non-degenerate. 
Observe that the refined approximate metric is close to 
the na\"ive approximate metric in the $C^{2,\alpha}_{\delta_0;t}$
weighted norm, for $t$ sufficiently small. 
From identical arguments as in \cite[Section~8]{Biquard}, 
for any metric $g_E$ satisfying \eqref{t1n}, then there exists a 
diffeomorphism $\varphi : M^t \rightarrow M^t$ so that $B_{g_t} ( g_t - \varphi^* g_E) = 0$,
and $\varphi^* g_E$ remains close to $g_t$ in the $C^{2,\alpha}_{\delta_0;t}$ weighted norm.  
The metric $\varphi^* g_E$ would then give a zero of the mapping $\Phi_t$, 
but then \eqref{lexp} would then imply that $\lambda = 0$, a contradiction.

If $g_0$ is not non-degenerate, then there is a finite 
dimensional space of infinitesimal Einstein deformations
on $(M_0,g_0)$, of dimension $k \geq 0$. 
The proof of Proposition \ref{obs} can be modified to 
yield a mapping 
\begin{align}
\psi_t: B_1 \times (-\epsilon, \epsilon) \rightarrow  B_1 \times (-\epsilon, \epsilon),
\end{align}
where $B_1$ is a small ball in $\mathbb{R}^k$,  whose zero set is
in one-to-one correspondence with the zeroes of $\Phi_{g_t}$
(see also the argument in \cite[Section 11]{GurskyViaclovsky}).
An almost identical argument as in the proof of Proposition \ref{obs}
shows that the mapping $\psi$ admits an expansion 
\begin{align}
\psi( a_1, \dots, a_k, v) = \big( \psi_1, \dots, \psi_k, t \lambda 
+ O(t^{\frac{3}{2} - \delta} ) \big), \end{align}
as $t \rightarrow 0$, so the existence of an Einstein metric in the 
stated neighborhood still implies that $\lambda =0$.  
\end{proof}

\section{Analysis on asymptotically hyperbolic manifolds}
\label{AHEsec}
The goal of this section is to study the asymptotic behavior of $L^{2}$-solutions of some elliptic operators on an AHE manifold $(M^{m},\p M,g,x)$. 
We consider the following operators:

\begin{enumerate}
	\item $P_{0}=\Delta_{H}-2Ric$, acting on functions,\\
\item $P_{1}= B \delta^* =\frac{1}{2}(\nabla^{\ast}\nabla-Ric)=\frac{1}{2}(\Delta_{H}-2Ric)$, acting on $1$-forms,\\
\item $P_{2}=\Delta_{H}-2Ric$, acting on $2$-forms,\\
\item $P=\frac{1}{2}\nabla^{\ast}\nabla-\ring{R}$, acting on traceless 
symmetric $2$-tensors.
\end{enumerate}
We also need the following decompositions, noting that we will 
only consider $C^{\infty}$ sections of bundles, so notations of regularity will be 
suppressed. 

For $\res{\Omega^{1}{\overbar{M}}}_{\p M}$, we have the decomposition
\begin{align}
	\res{\Omega^{1}{\overbar{M}}}_{\p M}=\res{\Omega^{1}_{t}\overbar{M}}_{\p M}\oplus \res{\Omega^{1}_{n}\overbar{M}}_{\p M},
\end{align}
where $\Omega^{1}_{t}(\overbar{M})$ is the space of one-forms
on $\overbar{M}$ which vanish on normal vectors (with respect to $\overbar{g}$) 
when restricted to the boundary and $\Omega^{1}_{n}(\overbar{M})$ is the space one-forms on $\overbar{M}$ 
which vanish on tangential vectors when restricted to the boundary.

Next, for  $\res{\Omega^{2}{\overbar{M}}}_{\p M}$ we have the following,
\begin{align}
	\res{\Omega^{2} \overbar{M}}_{\p M}=\res{\Omega_{t}^{2} \overbar{M}}_{\p M} \oplus \res{(\Omega^{1}_{n}(\overbar{M})}_{\p M} \wedge \res{\Omega^{1}_{t} \p \overbar{M})}_{\p M},
\end{align}
where $\Omega^{2}_{t}(\overbar{M})$ is space of two-forms on $\overbar{M}$ 
which satisfy $ \overbar{\p_{x}}\lrcorner \omega = 0$ on the boundary,
where $\overbar{\p_{x}}$ is dual of $dx$ with respect to $\bar{g}$.

For $\res{S^{2}_{0}(\overbar{M})}_{\p M}$ we have the following decomposition
\begin{align}
	\res{S^{2}_{0}(\overbar{M})}_{\p M}=V_{0}\oplus V_{1} \oplus V_{2},
\end{align}
where 
\begin{align}
\label{tensors}
\begin{split}
&V_{0}\equiv\{h\in \res{S^{2}_{0}(\overbar{M})}_{\p M} \ | \ tr_{\overbar{g}}h=0, \overbar{\p_{x}}\lrcorner h=0\},\\
&V_{1}\equiv\{h\in \res{S^{2}_{0}(\overbar{M})}_{\p M} \ | \ h=v \odot dx, v\in T^{\ast}\p M\},\\
&V_{2}\equiv\{h\in \res{S^{2}_{0}(\overbar{M})}_{\p M} \ | \ h
\ \mbox{is a multiple of }  ((m-1) dx \odot dx-\overbar{g})\},
\end{split}
\end{align}
where $\odot$ denotes the symmetric product. 
Finally, let $\overbar{V}_{0},\overbar{V}_{1},\overbar{V}_{2}$ be those traceless symmetric tensors on $\overbar{M}$ such that they lie in $V_{0},V_{1},V_{2}$ when restricted to the boundary, respectively. 

  An important property of these operators that they are uniformly degenerate at the boundary, and there are a finite set of numbers, called {\textit{indicial roots}}
which characterize invertibility of the associated indicial operators
\cite{Leebook, Mazzeo}.  
The following proposition gives the indicial roots of these operators, and will be used throughout the remainder of the paper. For simplicity, we state the result for solutions defined on all of $M$, but note that the same conclusions hold only assuming that the solutions are defined in a neighborhood of the boundary of $M$. 
\begin{prop}
\label{weights}
	Let $(M^{m},\p M,g ,x)$ be an AH manifold. 
\begin{itemize}
\item
The indicial roots associated with $P_{0}$ are given by\\
	\begin{align}
		\delta_{0,P_0}^{+},\delta_{0,P_0}^{-}=\frac{m-1\pm \sqrt{(m-1)^{2}+8(m-1)}}{2}.	
	\end{align}
If $P_{0}(f)=0$ for some function $f \in L^2(M)$ 
then $f= f_{\infty} x^{\delta_{0,P_{0}}^{+}}+o(x^{\delta_{0,P_{0}}^{+}})$ 
as $x \rightarrow 0$, where $f_{\infty} : \overbar{M} \rightarrow \R$,  
and if $f_{\infty}=0$ on $\p M$, then $f \equiv 0$.\\
\item
The indicial roots associated with $P_{1}$ are given by\\
\begin{align}
\begin{split}
	&\delta_{0,P_1}^{+},\delta_{0,P_1}^{-}=m,-1\\
	&\delta_{1,P_1}^{+},\delta_{1,P_1}^{-}=\frac{m-1\pm \sqrt{(m-1)^{2}+8(m-1)}}{2},
\end{split}
\end{align}
where the roots $\delta^{\pm}_{0,P_1}$ correspond to sections of $\Omega^{1}_{t}\overbar{M}$ and $\delta_{1,P_1}^{\pm}$ corresponds to sections $\Omega^{1}_{n}\overbar{M}$. Moreover, if
$\alpha \in \Omega^1(M)$ in is $L^2$ and solves $P_{1}(\alpha) = 0$, then we have the following expansion as $x \rightarrow 0$
\begin{align}
\alpha=\alpha_{t}x^{\delta_{0,P_{1}}^{+}}+o(x^{\delta_{0,P_{1}}^{+}}),
\end{align}
where $x\alpha_{t}\in \Omega^{1}_{t}\overbar{M}$ and if $x \alpha_{t}\equiv 0$
on $\partial M$ then 
\begin{align}
\alpha=\alpha_{n}x^{\delta_{1,P_{1}}^{+}}+o(x^{\delta_{1,P_{1}}^{+}}),
\end{align}
where $x\alpha_{n} \in \Omega^{1}_{n}\overbar{M}$ and if $x \alpha_{n}\equiv 0$ on $\partial M$ then $\alpha\equiv 0$.\\
\item
The indicial roots associated with $P_{2}$ are given by\\
\begin{align}
\begin{split}
&\delta_{0,P_2}^{+},\delta_{0,P_2}^{-}=\frac{m-1\pm \sqrt{(m-1)^{2}+16}}{2}\\
&\delta_{1,P_2}^{+},\delta_{1,P_2}^{-}=m,-1,
\end{split}
\end{align}
where $\delta^{\pm}_{0,P_2}$ corresponds to sections of $\Omega^{2}_{t}\overbar{M}$ and $\delta_{1,P_2}^{\pm}$ corresponds to sections of $(\Omega^{1}_{n} \overbar{M} \wedge \Omega^{1}_{t} \overbar{M})$. Moreover, if $\omega \in \Omega^2(M)$ is in $L^2$ and 
satisfies $P_{2}\omega = 0$, then we have the following expansion as $x \rightarrow 0$
\begin{align}
\omega=\omega_{t}x^{\delta_{0,P_{2}}^{+}}+o(x^{\delta_{0,P_{2}}^{+}})
\end{align}
where $x^{2}\omega_{t}\in \Omega^{2}_{t}\overbar{M}$ and if 
$x^2 \omega_{t}\equiv 0$ on $\partial M$ then 
\begin{align}
\omega=\omega_{n}x^{\delta_{1,P_{2}}^{+}}+o(x^{\delta_{1,P_{2}}^{+}}),
\end{align}
where $x^{2}\omega_{n}\in (\Omega^{1}_{n} \overbar{M} \wedge \Omega^{1}_{t} \overbar{M})$ and if $x^2 \omega_{n}\equiv 0$ on $\partial M$ then $\omega \equiv 0$.\\
\item
The indicial roots associated with $P$ are given by\\
\begin{align}
\begin{split}
&\delta_{0,P}^{+},\delta_{0,P}^{-}=m-1,0\\
&\delta_{1,P}^{+},\delta_{1,P}^{-}=m,-1\\
&\delta_{2,P}^{+},\delta_{2,P}^{-}=\frac{m-1\pm \sqrt{(m-1)^{2}+8(m-1)}}{2},
\end{split}
\end{align}
where $\delta_{0,P}^{\pm},\delta_{1,P}^{\pm},\delta_{2,P}^{\pm}$ corresponds to sections of $V_{0},V_{1},V_{2}$, respectively. Moreover, if $h \in S^2_0 (T^*M)$ is in $L^2$ 
and satisfies $Ph = 0$, then we have the following expansion as $x \rightarrow 0$
\begin{align}
h=\sigma_{0}x^{\delta_{0,P}^{+}}+o(x^{\delta_{0,P}^{+}})
\end{align}
where $x^{2}\sigma_{0}\in \overbar{V}_{0}$ and if $x^2 \sigma_{0}\equiv 0$ on $\partial M$ 
then 
\begin{align}
h=\sigma_{1}x^{\delta_{1,P}^{+}}+o(x^{\delta_{1,P}^{+}}),
\end{align}
where $x^{2}\sigma_{1} \in \overbar{V}_{1}$ and if $x^2 \sigma_{1}\equiv 0$ on $\partial M$ then 
\begin{align}
h=\sigma_{2}x^{\delta_{2,P}^{+}}+o(x^{\delta_{2,P}^{+}})
\end{align}
where $x^{2}\sigma_{2} \in \overbar{V}_{2}$ and if $x^2 \sigma_{2}\equiv 0$ on $\partial M$ then $h\equiv 0$.	
\end{itemize}
\end{prop}
\begin{proof}
The proof of existence of an expansion around infinity for uniformly degenerate elliptic operators can be found in \cite{MazzeoHodge, Mazzeo}. The vanishing results follow from a unique continuation principle due to Mazzeo \cite{Mazzeo1, Mazzeo}. The calculation of indicial roots can be found in \cite{Lee, Leebook, MazzeoHodge}. 
\end{proof}
The reader will notice that there are several coincidences between the indicial roots in Proposition \ref{weights}. This is not an accident, due the the following relations between the various operators (which also contains some formulas which will be useful later). 
\begin{lem} 
\label{operators}
Let $(M^{m},g)$ be an Einstein manifold such that $Ric_{g}=\Lambda g$ then we have the following relations for $f: M \rightarrow \R$, $\omega \in \Omega^1(M)$, 
and $h \in \Gamma(S^2 (T^*M))$. 
\begin{enumerate}
\item
If $P_0(f) =0$ then $P_1 (df) = 0$.
\item
$P ( f \cdot g) = \frac{1}{2}P_0(f) \cdot g$.
\item
If $P_1(\omega) = 0$ then $P_0 (\delta \omega) = 0$ and $P_2 (d \omega) = 0$. 
\item
If $Ph = 0$ then $P_1(B h) = 0$ and $P_0 ( \delta B h) = 0$.  
\item
$P\delta^{\ast}=\delta^{\ast}B\delta^{\ast}$.
\end{enumerate}
\end{lem}
\begin{proof}
The proof is elementary, and is omitted. 
\end{proof}
\section{From germs to boundary tensors}
\label{sectionk}
To complete the proof of Theorem \ref{theorem2}, we will need to 
find global solutions of $P h = 0 $ on $M$ with prescribed
leading order quadratic asymptotics at the point $p_0$, 
and such that $h$ is bounded (in the interior metric) on approach to conformal infinity. 
The argument is based on the four-dimensional case 
considered by Biquard in \cite{Biquard}, but there are some key differences. We have the following definition.
\begin{deff}{\em
Let $(M^m, \p M,g,x)$ be an AH orbifold with one singular point $p_{0}$. Let $F_{k}$ denote the space of harmonic polynomials of degree $k$, in normal coordinates at $p_0$. 
We refer to elements of $F_k \otimes S^{2}_{0}(T^*_{p_{0}}M)$ as $k$-germs at $p_0$.
}
\end{deff}
We next need a few basic lemmas.
\begin{lem}
\label{lemm}
Let $(M^{m},\p M,x,g)$ be an AH manifold and let $h$ be a symmetric two tensor defined in neighborhood of the boundary. Then there exist a one-form $\alpha$ defined in neighborhood of the boundary such that $s=h+\delta^{\ast}\alpha$ is in radial gauge (i.e. $\p_{x}\lrcorner s=0$).
Furthermore, if $|h|_g =o(x^{m-1})$,  
then $\alpha$ can be chosen to satisfy 
$|\alpha|_g = o(x^{m-1})$ as $x \rightarrow 0$.
\end{lem}
\begin{proof}
This is easily proved using a special boundary defining function $x$ such that
\begin{align}
	g=\sinh^{-2}(x)(dx^{2}+g_{x}),
\end{align}
where $g_{x}$ is an $x$-dependent metric on a slice, and $g_{0}=\overbar{g}$.
This is proved, for example, in \cite{Lars}. The remainder of the proof is a standard calculation, and is omitted. 
\end{proof}
We also need the following analytic continuation result for analytic Riemannian manifolds.
\begin{lem}
\label{analytic}
Let $(M^m, \p M, g,x)$ be AHE and let $P(s)=0$ in some open 
set $V \subset M$ such that $s=\delta^{\ast} \alpha$ for some one-form $\alpha$ defined in some open ball $U \subset V$ then $\alpha$ can be analytically continued to a vector field $\tilde{\alpha}$ defined on any simply connected domain $U\subset U_{1} \subset V$ such that $s=\delta^{\ast}\tilde{\alpha}$. Furthermore, if
$V$ contains a punctured ball around $p_0$, and 
$s=O(|z|^{\gamma})$ for $\gamma \neq -1$ where $|z|$ is distance to the point $p_0$, 
then $\tilde{\alpha}=O(|z|^{\gamma+1})$ as $|z| \rightarrow 0$. 
\end{lem}
\begin{proof}
Since Einstein metrics are real analytic in harmonic coordinates one can use the infinitesimal version of analytic continuation technique (see \cite[Corollary~6.4]{kob}) to extend $\alpha$ to a larger domain. We refer to \cite[Lemma 2.6]{ander}  for a complete proof.
Finally, an examination of the proof shows that the solution is obtained by integration, and the last statement follows from this (the details are omitted).  
\end{proof}

Another crucial result we will need is the following unique continuation theorem due to Biquard.
\begin{thm}[Biquard \cite{biquniq}]
\label{uniquecon}
Let $(M^{m},\p M,g,x)$ be AHE. If $h \in \Gamma( S^2(T^*U))$,
where $U$ is a neighborhood of $\p M$,  
is a solution of $d_g Ric (h) = \Lambda h$,
and $h$ is in radial gauge (i.e $\p_{x} \lrcorner h=0$), 
such that $\abs{h}_{g}=o(x^{m-1})$ as $x\to 0$, then $h\equiv 0$.
\end{thm}

We next define a mapping on $0$-germs
\begin{align}
S_0^{\infty} : S^2_0(T^*_{p_0}M_0)   \to V_{0},
\end{align}
where $V_0$ was defined in \eqref{tensors}. 
Given a traceless symmetric matrix $a_{ij}$,
then there is a unique solution 
$\tilde{h} \in \Gamma(S^2_0(T^*M_0))$ of $P( \tilde{h}) = 0 $ satisfying 
\begin{align}
\tilde{h} = \frac{a_{ij}}{|z|^{m-2}} dz^i \odot dz^j + O(|z|^{3-m})
\end{align}
as $|z|  \rightarrow 0$, and such that 
$\tilde{h} \in L^2 (M_{0}\setminus B(p_{0},\epsilon))$.
The solution is unique since $M_0$ is assumed to be non-degenerate. 
From Proposition \ref{weights},  $\tilde{h}$ admits an expansion  
\begin{align}
\tilde{h} = \sigma_0 x^{m-1}  + o(x^{m-1}), 
\end{align}
with $x^2 \sigma_0 \in \overbar{V}_0$,
as $x \rightarrow 0$.  Since $\tilde{h}$ is unique,
the mapping 
\begin{align}
S_0^{\infty} (a_{ij} dz^i \odot dz^j) \equiv x^2 \sigma_0|_{\partial M_0}
\end{align}
is well-defined. Next, let $E_{0}=Im(S_{0}^{\infty})$. 
In a similar fashion we can define a mapping on $1$-germs
\begin{align}
S_{1}^{\infty}:F_1 \otimes S^{2}_{0}(T^*_{p_0} M_0) \to E_{0}^{\perp} \subset V_0,
\end{align}
with respect to the $L^2$ inner product on $\p M_0$, defined by the following. Given
\begin{align}
a_{ijk}z^{i}dz^{j}\odot dz^{k}\in F_1 \otimes S^{2}_{0}(T^*_{p_{0}}M_0 )	
\end{align}
then there is a solution 
$\tilde{h} \in \Gamma ( S^2_0 (T^*M_0))$ of $P( \tilde{h}) = 0 $ satisfying 
\begin{align}
\tilde{h} = \frac{a_{ijk}z^i}{|z|^{m}} dz^j \odot dz^k + O(|z|^{2-m})
\end{align}
as $|z|  \rightarrow 0$, and such that 
$\tilde{h} \in L^2 (M_0 \setminus B(p_0,\epsilon))$. Again, from 
Proposition \ref{weights} there is an expansion 
\begin{align}
\tilde{h} = \sigma_0 x^{m-1}  + o(x^{m-1}), 
\end{align}
where $x^2 \sigma_0 \in \overbar{V}_0$, as $x \rightarrow 0$.  Since $\tilde{h}$ is only unique 
up to the solutions from the $S_0^{\infty}$ construction above, the mapping 
\begin{align}
S_1^{\infty} (a_{ijk} z^i dz^j \odot dz^k) = x^2\sigma_0|_{\p M_0}
\end{align}
is well-defined modulo the image of $S_0^{\infty}$ and therefore gives a well-defined mapping to $E_{0}^{\perp}$. Letting $E_1 = Im(S^{\infty}_1)$,
in a similar fashion, we may define a mapping on $2$-germs
\begin{align}
S_2^{\infty} : S^2_0(\R^{m}) \otimes S^{2}_{0}(T^*_{p_{0}}M_0)  \to (E_{0}\oplus E_{1})^{\perp}.
\end{align}
by the following. Given 
\begin{align}
a_{ijkl} z^i z^j dz^k\odot dz^l \in F_2 \otimes  S^2_0(T^*_{p_{0}}M_0),
\end{align}
then there is a solution 
$\tilde{h} \in \Gamma ( S^2_0 (T^*M_0))$ of $P( \tilde{h}) = 0 $ satisfying 
\begin{align}
\tilde{h} = \frac{a_{ijkl}z^i z^j}{|z|^{m+2}} dz^k \odot dz^l + O(|z|^{1-m})
\end{align}
as $|z|  \rightarrow 0$, and such that 
$\tilde{h} \in L^2 (M_0 \setminus B(p_0,\epsilon))$. Again, from 
Proposition \ref{weights} there is an expansion 
\begin{align}
\tilde{h} = \sigma_0 x^{m-1}  + o(x^{m-1}), 
\end{align}
where $x^2 \sigma_0 \in \overbar{V}_0$ as $x \rightarrow 0$.  
Since $\tilde{h}$ is only unique 
up to the solutions from the $S_0^{\infty}$ and $S_{1}^{\infty}$ constructions above, the mapping 
\begin{align}
S_2^{\infty} (a_{ijkl}) = x^2 \sigma_0|_{\p M_0} 
\end{align}
is well-defined modulo the image of $S_0^{\infty}$ and $S_{1}^{\infty}$ and therefore gives a well-defined mapping to $(E_{0}\oplus E_{1})^{\perp}$. 
\begin{prop}
\label{injprop}
Let $(M_0^{m},\p M_0,g,x)$ be an AHE orbifold with one singular point $p_{0}$ 
of type $\Gamma_n$($m=2n$) and non-degenerate. Then we have the following:
\begin{enumerate}
\item $S_{0}^{\infty}$ is injective.\\
 \item $S_{1}^{\infty}$ is injective.\\
\item The kernel of $S_{2}^{\infty}$ is the  
space of $2$-germs of the form $|z|^{m+2} \mathcal{K}_{euc}(\frac{\alpha_{1}}{|z|^{m}})$, 
where $\alpha_1$ is a one-form with linear coefficients.\footnote{
The reader should compare this statement with \cite[Lemma 10.1]{Biquard}, from which our conclusion differs.} 
\end{enumerate}
\end{prop}
\begin{proof}
In any of these cases, let $s$ be the solution from the above construction  
which is blowing up at the point, and $L^2$ in a neighborhood of infinity, so that top $L^{2}$ weight vanishes. By Proposition \ref{weights}, $s$ admits an expansion 
\begin{align}
s=\sigma_{1}x^{m}+o(x^{m})
\end{align}
where $x^2 \sigma_{1} \in \overbar{V}_{1}$, as $x \to 0$. 
A straightforward calculation shows that 
$B s = o(x^{m})$ as $x \to 0$. 
By Lemma \ref{operators}, $P_1(Bs) = 0$, so by Proposition~\ref{weights}, we must have 
\begin{align}
Bs = \alpha_n x^{\delta_{1,P_{1}}^+} + o( x^{\delta_{1,P_{1}}^+})
\end{align} 
where $\alpha_{n} \in \res{\Omega^{1}_{n}\overbar{M_0}}$, as $x \to 0$. 
Then, again by the Lemma~\ref{operators} and Proposition~\ref{weights} and the fact that $\delta_{1,P_{1}}^{+}>\delta_{1,P_{2}}^{+}$, we have that $d B s \equiv 0$. 
We need the following lemma.
\begin{lem} \label{smlem}
The function $f = \delta B s$ satisfies 
\begin{align}
\label{flem1}
B s &= \frac{1}{\Lambda} d f, \\
\label{flem2}
P_0 ( f) &= 0.
\end{align}
\end{lem}
\begin{proof}
Since $tr_g s = 0$ and $ d Bs = 0$,
using Lemma \ref{operators}, we have
\begin{align}
df = d \delta \delta s = (\Delta_H - \delta d) \delta s
= \Lambda \delta s  - \delta d \delta s
= \Lambda \delta s,
\end{align}
and \eqref{flem1} follows. Lemma \ref{operators} directly implies 
\eqref{flem2} since $Ps = 0$. 
\end{proof}
Letting $\tilde{f} = \frac{2 f}{(2-m) \Lambda}$, 
from Lemma \ref{smlem}, we have that 
\begin{align}
P ( s + \tilde{f} g ) = 0,\\
B ( s + \tilde{f} g ) = 0.
\end{align}
Consequently, 
\begin{align}
d_g Ric (  s + \tilde{f} g) &= \Lambda ( s + \tilde{f} g), \\
 s + \tilde{f} g &= o (x^{m-1}),
\end{align}
as $x \rightarrow 0$. 

Then by Theorem \ref{uniquecon} and Lemma \ref{lemm} it follows that
\begin{align}
 s + \tilde{f} g = \delta^{\ast} (\alpha) 
\end{align}
for some one-form $\alpha$ defined in a neighborhood of the boundary. 
Since $(M_0,g)$ is a non-compact and complete Riemannian manifold there exists a geodesic ray $\gamma$ starting from $p_{0}$ to infinity. Now applying Lemma \ref{analytic} inductively on $\gamma$ we can analytically continue $\alpha$ to a one-form $\tilde{\alpha}$ defined in a neighborhood of $\gamma$ such that
\begin{align}
 s + \tilde{f} g = \delta^{\ast} (\tilde{\alpha}).
\end{align}
Equivalently,
\begin{align}
\label{tilde}
s=\mathcal{K}(\tilde{\alpha}),
\end{align}
for a $1$-form $\tilde{\alpha}$ satisfying
\begin{align}
P_{1}(\tilde{\alpha})=0.
\end{align}

In the case of $S_{0}^{\infty}$, let $s$ 
be the solution of $Ps=0$ whose leading term is 
$|z|^{2-m}h_{0}$ at the point, where $h_0$ is a $0$-germ at $p_0$, and $s=o(x^{m-1})$ as $x\to 0$. Since $s=O(|z|^{-m+2})$ around $p_{0}$ and \eqref{tilde} is satisfied, 
by Lemma \ref{analytic}, $\tilde{\alpha}=O(|z|^{-m+3})$. 
The Bochner formula implies that there is no nontrivial $L^2$ kernel for $P_1$.
Since $P_1 \tilde{\alpha} = 0$, this implies that $\tilde{\alpha} = 0$ 
by a standard removable singularity theorem for solutions of $P_1$,
and therefore $h_0 = 0$. 
 
In case of $S_{1}^{\infty}$, given any $1$-germ $h_1$ at $p_0$
let $s$ be the solution of $Ps=0$ 
whose leading term is $\abs{z}^{-m}h_{1}$ at $p_0$, 
and $x^2 s|_{\p M_0} \in E_0^{\perp}$. If $h_1 \in Ker(S_1^{\infty})$ then $s=o(x^{m-1})$ as $x\to 0$.
Since $s=O(\abs{z}^{-m+1})$ as $\abs{z}\to 0$ and \eqref{tilde} is satisfied, by Lemma \ref{analytic}, $\tilde{\alpha}=O(\abs{z}^{-m+2})$. This implies that $\tilde{\alpha} = 0$ by the removable singularity theorem, 
since there is no solution with that leading term 
which is also invariant under group action, and therefore $h_1 = 0$.

In case of $S_{2}^{\infty}$, 
 given any $2$-germ $h_2$ at $p_0$
let $s$ be the solution of $Ps=0$ 
whose leading term is $\abs{z}^{-m-2}h_{2}$ at $p_0$, 
and $x^2 s|_{\p M_0} \in (E_0 \oplus E_1)^{\perp}$. If $h_2 \in Ker(S_2^{\infty})$ then $s=o(x^{m-1})$ as $x\to 0$.
Since $s=O(\abs{z}^{-m+1})$ as $\abs{z}\to 0$ and \eqref{tilde} is satisfied by Lemma \ref{analytic}, $\tilde{\alpha}=O(\abs{z}^{-m+2})$. This implies that
\begin{align}
\tilde{\alpha}=\frac{\alpha_{1}}{\abs{z}^{m}}+O(\abs{z}^{-m+2}),
\end{align}
where $\alpha_{1}$ is a one-form with linear coefficients. Then we have
\begin{align}
s=\mathcal{K}(\tilde{\alpha})=\mathcal{K}_{euc}\Big(\frac{\alpha_{1}}{\abs{z}^{m}}\Big)+O(\abs{z}^{-m+1})
\end{align}
as $|z| \to 0$. For the converse, since there is no nontrivial $L^2$ kernel for $P_1$, 
given a linear form $\alpha_1 = a_{ij} z^i dz^j$, there exists a unique 
solution $\tilde{\alpha}_1$ of $P_1 ( \tilde{\alpha}_1 ) = 0$
which is in $L^2( M_0 \setminus B(p_0, \epsilon))$
satisfying 
\begin{align}
\tilde{\alpha}_1 = |z|^{2-m}  a_{ij} z^i dz^j + O(|z|^{2-m}),
\end{align}
as $|z| \to 0$. By Proposition \ref{weights}, $\tilde{\alpha}_1$
admits an expansion 
\begin{align}
\tilde{\alpha}_1 = \alpha_{t}x^{\delta_{0,P_{1}}^{+}}+o(x^{\delta_{0,P_{1}}^{+}}),
\end{align}
where $x\alpha_{t}\in \Omega^{1}_{t}\overbar{M_0}$, as $x \to 0$.
From Proposition \ref{operators}, it follows easily that 
$h = \mathcal{K} \tilde{\alpha}_1$ 
is a solution of $Ph= 0$, and thus the $2$-germ
$|z|^{m+2} \mathcal{K}_{euc}(\frac{\alpha_{1}}{|z|^{m}})$, 
is in the kernel of $S^{\infty}_2$, where $\alpha_1$ is any 
one-form with linear coefficients.
\end{proof}
\section{From boundary tensors to germs}
\label{sectionl}
The following lemma describes a ``duality'' 
between boundary tensors and germs. We state the following 
\begin{lem}
\label{duality}
Let $(M^{m},\p M,g,x)$ be AH with one orbifold point $p_{0}$
of type $\Gamma_n$. Let $\sigma_{k},\tau_{k}$ be $k$-germs at $p_{0}$ and Let $s$ and $t$ be two solutions for $P$ such that
\begin{align}
\begin{split}
&s= \sigma_{k}\abs{z}^{2-m-2k}+ o(z^{2-m-k}) \text{ as } \abs{z}\to 0\\
&s=\sigma_{\infty}x^{\delta_{0,P}^{+}}	+o(x^{\delta_{0,P}^{+}}) \text{ as } x\to 0,\end{split}
\end{align}
where $x^{2}\sigma_{\infty}\in \overbar{V}_{0}$, and 
\begin{align}
\begin{split}
&t= \tau_{k} + o(\abs{z}^{k}) \text{ as } \abs{z}\to 0\\
&t= \tau_{\infty}x^{\delta_{0,P}^{-}}+o(x^{\delta_{0,P}^{-}}) \text{ as } x\to 0,
\end{split}
\end{align}
where $x^{2}\tau_{\infty}\in \overbar{V}_{0}$. For $k=0$ we have
\begin{align}
\label{k=0}
	(m-2) \frac{\omega_{m-1}}{n}\innn{\sigma_{0},\tau_{0}}=(m-1)\inn{x^2 \sigma_{\infty},x^2\tau_{\infty}},
\end{align}
for $k=1$ we have
\begin{align}
\label{k=1}
	\frac{\omega_{m-1}}{n}\innn{\sigma_{1},\tau_{1}}=(m-1)\inn{x^2 \sigma_{\infty},x^2 \tau_{\infty}},
\end{align}
and for $k=2$ we have
\begin{align}
\label{k=2}
	\frac{2\omega_{m-1}}{mn}\innn{\sigma_{2},\tau_{2}}=(m-1)\inn{x^2 \sigma_{\infty},x^2 \tau_{\infty}},
\end{align}where the left inner product is the germ inner product 
at $p_{0}$ (see the proof for the definition), 
and the right inner product is the $L^{2}$ inner product on the boundary. 
\begin{proof}
Let $M_{0,\epsilon}=\overbar{M}_0\setminus \big{(}B(p_0,\epsilon)\cup \p M_0\times [0,\epsilon)\big{)}$. Integrating by parts on $M_{0,\epsilon} $ implies that
\begin{align}
	\int_{M_{0,\epsilon}}(\inn{Ps,t}-\inn{s,Pt})dV_g=\int_{\p M_{0,\epsilon}}(\inn{-\nabla_{\vec{n}}s,t}+\inn{s,\nabla_{\vec{n}}t})dS_g.
	\end{align}
The interior boundary integral is given by 
\begin{align}
	\lim_{\epsilon \to 0} \int_{\p B(p_{0},\epsilon)}(\inn{-\nabla_{\vec{n}}s,t}+\inn{s,\nabla_{\vec{n}}t})dS_g = (2-m-2k) \int_{S^{m-1}/\Gamma_n} \inn{\sigma_k, \tau_k } dS_{euc},
	\end{align}
where the inner product denotes contraction on only the 
indices of the germs corresponding to $S^2_0(T^*_{p_0}(M_0))$.

Recall from Proposition \ref{weights} that $\delta_{0,P}^{+} = m-1$ 
and $\delta_{0,P}^{-}=0$. Since $\sqrt{\det \overbar{g}}=x^{1-m}\sqrt{\det g}$ when we restrict both metrics 
to $\p M_0 \times \{ \epsilon\}$, it follows that the outer boundary integral is given by 
\begin{align}
\lim_{\epsilon \to 0} \int_{ \p M_0 \times \{\epsilon \}}
(\inn{-\nabla_{\vec{n}}s,t}+\inn{s,\nabla_{\vec{n}}t})dS_g
= (m-1)\int_{\p M_0} \langle x^2 \sigma_{\infty},
x^2 \tau_{\infty} \rangle dS_{g_{\infty}}. 
\end{align}
If $k=0$ then then working in a coordinate chart around $p_{0}$ with $g_{ij}(p_{0})=\delta_{ij}$ we have
\begin{align}
		\sigma_{0}=H_{kl}dz^{k}\otimes dz^{l},\\
		\tau_{0}=K_{kl}dz^{k}\otimes dz^{l},
	\end{align}
and \eqref{k=0} follows immediately, 
with $\innn{ \sigma_0, \tau_0 } = \sum_{k,l} H_{kl}K_{kl}$. 

If $k=1$, then  
\begin{align}
		\sigma_{1}=H_{jkl}z^{j}dz^{k}\otimes dz^{l},\\
		\tau_{1}=K_{jkl}z^{j}dz^{k}\otimes dz^{l},
	\end{align}
and \eqref{k=1} follows from \eqref{sb1} with 
$\innn{\sigma_1, \tau_1 } = \sum_{j,k,l} H_{jkl}K_{jkl}$.

Finally, if $k=2$, then
	\begin{align}
		\sigma_{2}=H_{ijkl}z^{i}z^{j}dz^{k}\otimes dz^{l},\\
		\tau_{2}=K_{ijkl}z^{i}z^{j}dz^{k}\otimes dz^{l},
	\end{align}
and \eqref{k=2} follows from \eqref{sb2} 
with $\innn{ \sigma_2, \tau_2 } = \sum_{i,j,k,l} H_{ijkl}K_{ijkl}$.
\end{proof}
\end{lem}
Next, we explain how to construct a map which sends boundary tensors to germs at $p_{0}$. Identify a neighborhood of the boundary with $\p M_0 \times [0,\epsilon)$. Let $\varphi$ be a cutoff function which is $1$ on $\p M_0\times [0,\frac{\epsilon}{2})$ and is zero on 
$M_0 \setminus \big{(}\p M_0 \times [0,\epsilon)\big{)}$. 
Given $\sigma_0 \in V_0$, consider the tensor 
$\varphi x^{-2} \sigma_{0}$. By solving a finite number 
of terms of a power series, we can then find $s$ of the form 
$s = \varphi x^{-2}\sigma_0 + \tilde{s}$, where $\tilde{s}$ are 
lower order terms supported in a neighborhood of the boundary, 
such that $Ps$ is in $L^2$ in a neighborhood of the boundary. 
(For this power series method, see for example  \cite{Guill} for the Hodge Laplacian acting on forms; the same technique applies to solutions of $P$.)
Then, using the non-degeneracy assumption, we can find a solution 
of $Ps = 0 $ defined on all of $M_0$, which is in $L^2$ on any compact subset of $M_0$,
so that
\begin{align}
s = x^{-2} \sigma_0 + o(1)
\end{align}
as $x \rightarrow 0$.  Note that by non-degeneracy, the solution $s$ is unique. 
The solution $s$ admits the following expansion around $p_0$
\begin{align}
    s=h_{0}+o(1) \text{ as } \abs{z}\to 0,
    \end{align}
where $h_0$ is a $0$-germ at $p_0$. Define a mapping 
\begin{align}
\begin{split}
&S_{\infty}^0: V_0 \rightarrow S^2(T^*_{p_0}M_0), \mbox{ by}\\
&S_{\infty}^{0}(\sigma_{0})\equiv h_{0}. 
\end{split}
\end{align}
This mapping enjoys the following property.
\begin{lem}
\label{s0ilem}
The mapping $S_{\infty}^{0}: V_0 \rightarrow S^2(T^*_{p_0}M_0)  $ is surjective.
\end{lem}
\begin{proof}
Assume $S_{\infty}^{0}$ is not surjective. Then there exist a $0$-germ $h_0$ such that $h_0$ is orthogonal to the image of $S_{\infty}^{0}$. Let $\tilde{h}$ be the solution of $P \tilde{h} = 0$ with following leading terms
\begin{align}
&\tilde{h}=h_0 \abs{z}^{2-m}+o(\abs{z}^{2-m}) \text{ as } \abs{z}\to 0, \\
&\tilde{h}=S_{0}^{\infty}(h)x^{\delta_{0,P}^{+}}+o(x^{\delta_{0,P}^{+}}) \text{ as } x\to 0.
\end{align}
Also for a boundary tensor $\tau \in V_{0}$, let $\tilde{\tau}$ be the solution of $P \tilde{\tau} = 0$ with following leading terms
\begin{align}
\begin{split}
&\tilde{\tau}=S_{\infty}^{0}(\tau)+o(1) \text{ as } \abs{z} \to 0,\\
&\tilde{\tau}= x^{-2} \tau+o(1) \text{ as } x \to 0.	
\end{split}
\end{align}

By applying Lemma \ref{duality} to $\tilde{\tau}$ and $\tilde{h}$ we see that
\begin{align} 
\innn{ h_0, S^0_{\infty} \tau} = c \inn{S_{0}^{\infty}h_0, \tau},
\end{align}
for every $\tau \in V_0$, where $c \neq 0$, so $S_{0}^{\infty}(h)=0$. This implies that
$h_0 = 0$ by Proposition~\ref{injprop}~(1), which is a contradiction. 
\end{proof}

We next define a mapping from $E_0^{\perp} \subset V_0$  (the orthogonal 
complement in $L^2$) to $1$-germs by the following.  
Given $\sigma_0 \in E_0^{\perp}$, as above let $s$ be the solution of 
$Ps =0$ defined on all of $M_0$, which is in $L^2$ on any compact subset of $M_0$,
and satisfies 
\begin{align}
s = x^{-2} \sigma_0 + o(1), 
\end{align}
as $x \to 0$. By Lemma \ref{duality}, $s$ admits an expansion 
\begin{align}
s = h_1 + o(|z|),
\end{align}
as $|z| \to 0$, where $h_1$ is a $1$-germ at $p_0$. 
We then define the mapping
\begin{align}
&S_{\infty}^1: E_0^{\perp} \rightarrow F_1 \otimes S^2(T^*_{p_0}M_0), \mbox{ by}\\
&S_{\infty}^1(\sigma_0) \equiv h_1.
\end{align}
This mapping is also onto. 
\begin{lem}
\label{s1ilem}
The mapping $S_{\infty}^{1}: E_0^{\perp} \rightarrow F_1 \otimes S^2(T^*_{p_0}M_0)$ 
is surjective.
\end{lem}
\begin{proof}
By Proposition \ref{injprop} (2), $S^{\infty}_1$ is injective, 
so by the same argument in the proof of Lemma \ref{s0ilem},
it follows that $S_{\infty}^1$ is surjective. 
\end{proof}

We next define a mapping from $(E_0 \oplus E_1) ^{\perp} \subset V_0$  (the orthogonal 
complement in $L^2$) to $2$-germs by the following.  
Given $\sigma_0 \in (E_0 \oplus E_1)^{\perp}$, as above let $s$ be the solution of 
$Ps =0$ defined on all of $M_0$, which is in $L^2$ on any compact subset of $M_0$,
and satisfies 
\begin{align}
s = x^{-2} \sigma_0 + o(1), 
\end{align}
as $x \to 0$. By Lemma \ref{duality}, 
$s$ admits an expansion 
\begin{align}
s = h_2 + o(|z|^2),
\end{align}
as $|z| \to 0$, where $h_2$ is a $2$-germ at $p_0$. 

We then define the mapping
\begin{align}
\begin{split}
&S_{\infty}^2: (E_0 \oplus E_1)^{\perp} \rightarrow F_2 \otimes S^2(T^*_{p_0}M_0), \mbox{ by}\\
&S_{\infty}^2(\sigma_0) \equiv h_2.
\end{split}
\end{align}
The following Lemma says that $S^2_{\infty}$ is {\textit{not}} surjective, 
and also gives a complete characterization of its image\footnote{The 
reader should compare to \cite[Lemma 11.1]{Biquard}, where it is claimed that this mapping is surjective. However, this does not affect any of the main results in \cite{Biquard}, since one can simply gauge the elements $k_2$ and $k_3$ so that the germs of their leading terms are orthogonal to $S$. See also~\cite{Biquard3}.} 
\begin{lem}
\label{biglem}
Let $T= Im(S_{\infty}^{2})$ then
\begin{align}
T=S^{\perp},
\end{align}
where the orthogonal complement is with respect to the $2$-germ inner product,
and 
\begin{align}
 S=\Big\{ |z|^{m+2} \mathcal{K}_{euc}\Big( \frac{\alpha}{|z|^{m}} \Big) | \
\alpha \text{ is a one-form with linear coefficients} \Big\}.
\end{align}
\end{lem}
\begin{proof}
Let $h_2\in T$ so $h_2=S_{\infty}^{2}(\tau)$ and let $\tilde{h}$ be the solution of $P
\tilde{h} = 0$ with following leading terms
\begin{align}
\tilde{h}&=h_2 +o(\abs{z}^{2}) \text{ as } \abs{z} \to 0,\\
\tilde{h}&= x^{-2} \tau + o(1) \text{ as } x\to 0.		
\end{align}
For any $t_2 \in S$, Let $\tilde{t}$ be the solution of $P \tilde{t} = 0$ with 
the following leading terms
\begin{align}
\tilde{t}&=t_2 \abs{z}^{-2-m}+o(\abs{z}^{-m}) \text{ as } \abs{z}\to 0,\\
\tilde{t}&=S_{2}^{\infty}(t_2) x^{-2} x^{m-1} +o(x^{m-1}) \text{ as } x \to 0.
\end{align}
Applying Lemma \ref{duality} to $\tilde{t}$ and $\tilde{h}$ we obtain
\begin{align}
\innn{ h_2, t_2 } = \innn{S_{\infty}^{2}(\tau),t_2}
= c \inn{\tau ,S_{2}^{\infty}(t_2)}=0,
\end{align}
for some constant $c \neq 0$, so by Proposition \ref{injprop} it follows that $T\subset S^{\perp}$. 

Next, assume that $S_{\infty}^{2}$ is not surjective onto $S^{\perp}$. 
Then there exists a nontrivial $s_{2}\in S^{\perp}\cap T^{\perp}$. 
Let $\tilde{s}$ be the solution of 
$P(\tilde{s}) = 0$ with the  following leading terms
\begin{align}
\tilde{s}&=s_{2}\abs{z}^{-2-m}+o(\abs{z}^{-m}) \text{ as } \abs{z}\to 0,\\
\tilde{s}&=S_{2}^{\infty}(s_2) x^{-2} x^{m-1} +o(x^{m-1})\text{ as } x\to 0.\end{align}
Also, for any boundary tensor $ \tau \in (E_{0}\oplus E_{1})^{\perp}$, 
let $\tilde{t}$ be the solution of $P$ with following leading terms
\begin{align}
\tilde{t}&=S_{\infty}^{2}(\tau) +o(\abs{z}^{2}) \text{ as } \abs{z}\to 0,\\
\tilde{t}&= x^{-2} \tau +o(1) \text{ as } x\to 0.		
\end{align}
Applying Lemma \ref{duality} to $\tilde{t}$ and $\tilde{s}$ we obtain
\begin{align}
\inn{S_{2}^{\infty}(s_{2}),\tau} = c \innn{ s_2, S^2_{\infty}(\tau)} = 0, 
\end{align}
which implies that
$S_{2}^{\infty}(s_{2})=0$ so $s_{2} \in S$ by Proposition \ref{injprop} (3),
which gives a contradiction.
\end{proof}
Next, we give an explicit description of the elements of $S$.
\begin{lem}
\label{conformalkilling}
Let $\alpha=\frac{\alpha_{1}}{|z|^{m}}$ where $\alpha_{1}=a_{ij}z_{i}dz^{j}$ is a one-form with linear coefficients. Then
\begin{align}
\label{Kform}
\mathcal{K}_{euc}(\alpha)_{ij}=\frac{1}{2}\Big{(}\frac{a_{ij}+a_{ji}}{|z|^{m}}+\frac{-m z_{i}z_{p}a_{pj}-m z_{j}z_{t}a_{ti}}{|z|^{m+2}}\Big{)}
+\frac{1}{m}\Big{(}	\frac{-tr(a)}{|z|^{m}}+\frac{m z_{t}z_{p}a_{pt}}{|z|^{m+2}}\Big{)}\delta_{ij}
\end{align}
Furthermore, we may write
\begin{align} 
\label{KKform}
\mathcal{K}_{euc}(\alpha) = |z|^{-m-2} K_{klij} z^k z^l dz^i dz^j,
\end{align}
where 
\begin{align}
\begin{split}
\label{KKKform}
K_{ijkl}&=\frac{1}{2}\Big{(}(a_{ij}+a_{ji})\delta_{kl}-\frac{m}{2}(\delta_{ik}a_{lj}+\delta_{il}a_{kj})-\frac{m}{2}(\delta_{kj}a_{li}+\delta_{lj}a_{ki})\Big{)}\\
&+\frac{-tr(a)\delta_{kl}\delta_{ij}}{m}+\frac{\delta_{ij}}{2}(a_{lk}+a_{kl}).
\end{split}
\end{align}
\end{lem}
\begin{proof}
Let $\alpha_{1}=a_{ij}x_{i}dx^{j}$ be a one-form with linear coefficients. Then we have
\begin{align}
\begin{split}
(\delta_{euc}^{\ast} \alpha)_{ij}&=\frac{1}{2}(\p_{i}\alpha_{j}+\p_{j}\alpha_{i})\\
&=\frac{1}{2}\Big{(}\p_{i} \Big(\frac{a_{pj}z_{p}}{|z|^{m}}\Big)+\p_{j}\Big(\frac{a_{ki}z_{k}}{|z|^{m}}\Big)\Big{)}.
\end{split}
\end{align}
Using $\p_{i}(|z|^{-m})=\frac{-mz_{i}}{|z|^{m+2}}$ we obtain
\begin{align}
(\delta_{euc}^{\ast} \alpha)_{ij}=\frac{1}{2}\Big{(}\frac{a_{ij}+a_{ji}}{|z|^{m}}+\frac{-m z_{i}z_{p}a_{pj}-m z_{j}z_{k}a_{ki}}{|z|^{m+2}}\Big{)}.
\end{align}
Also,
\begin{align}
\delta_{euc}(\alpha)=-\p_{k}X_{k}=-\p_{k}\Big(\frac{a_{pk}z_{p}}{|z|^{m}}\Big)=\frac{-tr(a)}{|z|^{m}}+\frac{mz_{k}z_{p}a_{pk}}{|z|^{m+2}}.
\end{align}
The formula \eqref{Kform} 
follows directly from these computations, and \eqref{KKform} follows easily from this. 
\end{proof}
\begin{lem}
\label{comp2}
The tensor $r^{2}g_{euc} = P_{ijkl} z^i z^j dz^k dz^l$,
 where \begin{align}
P_{ijkl}=\delta_{ij}\delta_{kl}
\end{align} 
satisfies $B_{euc}(r^{2}g_{euc})=2(n-1)rdr$.
%
The tensor $r^{2}dr\otimes dr =  Q_{ijkl} z^i z^j dz^k dz^l$, where 
\begin{align}
Q_{ijkl}=\frac{1}{2} ( \delta_{ik}\delta_{jl}+\delta_{il}\delta_{jk}),
\end{align}
satisfies $B_{euc}(r^{2}dr\otimes dr)=(-2n)rdr$.
The tensor $r^{4}\theta \otimes \theta =   A_{ijkl} z^i z^j dz^k dz^l$, where
\begin{align}
A_{ijkl}=\frac{1}{2} ( J_{i}^{k}J_{j}^{l}+J_{i}^{l}J_{j}^{k}),
\end{align}
satisfies $B_{euc}(r^{4}\theta\otimes \theta)=2rdr$.
\end{lem}
\begin{proof}
The proof is a simple calculation, and is omitted.
\end{proof}
\begin{lem}
\label{comp1}
We have the following identities:
\begin{align}
&\innn{Q,K}=\frac{2-m^{2}-m}{2}tr(a),\\
&\innn{A,K}=\frac{m+2}{2}tr(a),\\
&\innn{K,P}=0.	
\end{align}
\end{lem}
\begin{proof}
We compute
\begin{align}
\begin{split}
\innn{Q,K}&=\frac{1}{2}\big{(}(a_{ij}+a_{ji})\delta_{ij}-\frac{m}{2}(\delta_{ii}a_{jj}+\delta_{ij}a_{ij})-\frac{m}{2}(\delta_{ij}a_{ji}+\delta_{jj}a_{ii})\big{)}\\
&+\frac{-tr(a)\delta_{ij}\delta_{ij}}{m}+\frac{\delta_{ij}}{2}(a_{ji}+a_{ij})\\
&=\frac{1}{2}\big{(}2tr(a)-m^{2}tr(a)-mtr(a))\big{)}=\frac{2-m^{2}-m}{2}tr(a).
\end{split}
\end{align}
This proves the first identity. For the second identity, we have
\begin{align}
\begin{split}
\innn{A,K}=&\frac{1}{2}\Big{(}(a_{ij}+a_{ji})\delta_{kl}J_{i}^{k}J_{j}^{l}-\frac{m}{2}(J_{i}^{k}J_{j}^{l}\delta_{ik}a_{lj}+J_{i}^{k}J_{j}^{l}\delta_{il}a_{kj})\\
&-\frac{m}{2}(J_{i}^{k}J_{j}^{l}\delta_{kj}a_{li}+J_{i}^{k}J_{j}^{l}\delta_{lj}a_{ki})\Big{)}
+\frac{-tr(a)\delta_{kl}\delta_{ij}J_{i}^{k}J_{j}^{l}}{m}+J_{i}^{k}J_{j}^{l}\delta_{ij}\frac{(a_{lk}+a_{kl})}{2}\\
&=\frac{1}{2}\Big{(}2tr(a)+mtr(a)\Big{)}-tr(a)+tr(a)=\frac{m+2}{2}tr(a).
\end{split}
\end{align}	
Finally, the last identity follows from the fact that $K$ is traceless on the first and last two pairs of indices.
\end{proof}
We next construct a special symmetric $2$-tensor 
$k_1$ on the Calabi manifold. 
\begin{prop}
\label{tensork33}
Let $\sigma_2$ denote the quadratic $2$-tensor
\begin{align}
\sigma_2&= r^2 (-dr\otimes dr-(2n-1)r^{2}\theta\otimes \theta + g_{euc}) = S_{ijkl} x^i x^j dx^k dx^l. 
\end{align}
Then 
\begin{align}
\label{kp1}
tr_{euc} (\sigma_2) = 0, \ B_{euc} (\sigma_2) =0 \ 
\Delta_{euc} (\sigma_2)  = 0, 
\end{align} and 
\begin{align}
\label{kp4}
\innn{ S, K } = 0,
\end{align}
for all $K_{ijkl}$ of the form in \eqref{KKKform}. 
Furthermore, there exist a constant $c \neq 0$ and a tensor $k_{1}$ defined on 
the Calabi manifold with the following properties
\begin{align}
\label{tensork3}
\begin{split}
&k_{1}=c \cdot  \sigma_2+ O(r^{-2n+2+\epsilon})\\
&P_{g_{cal}}(k_{1})=o\\
&B_{g_{cal}}(k_1) = 0,
\end{split}
\end{align}
as $r \rightarrow \infty$. 
\end{prop}
\begin{proof}
First, we have
\begin{align}
tr_{euc} ( \sigma_2) = r^2 \big( -1 - (2n-1) + 2n\big) = 0.
\end{align}
Next, using Lemma \ref{comp2}, we have
\begin{align}
B_{euc} (\sigma_2) = \big( 2n - 2(2n-1) + 2(n-1)\big) r dr = 0.
\end{align}
Next, it follows from  Lemma \ref{comp2} that $S_{ijkl} = S_{klij}$, 
so the first equation in \eqref{kp1} implies the last. 

To see \eqref{kp4}, Lemma \ref{comp1} implies that 
\begin{align}
\innn{ S, K } = \big( (n+1)(1-2n)-(1-2n^{2}-n) \big) tr(a) = 0.
\end{align}
Next, let $(H, \Lambda) = (\sigma_2, 0)$ in Proposition \ref{limitobs} to find a solution 
of 
\begin{align}
\label{oneone00}
\begin{split}
P_{g_{cal}}(k_1)&= \tilde{\lambda} o \\
k_1 &= c \cdot \sigma_2 +O(r^{-2n+2+\epsilon})\\
B_{g_{cal}}(h)&=0,
\end{split}
\end{align}
where $\tilde{\lambda}$ is given by 
\begin{align}
\label{obstr00}
\tilde{\lambda} = -\frac{c}{ \Vert o \Vert^{2}_{L^2}} \lim_{r\to\infty} \int_{S_{r}/\Gamma_{n}}\frac{n+1}{r}\inn{\sigma_2 ,o}dS_{{S_{r}}/\Gamma_{n}}.
\end{align}
The leading term in the integrand is given by 
\begin{align}
&\frac{n+1}{r} \langle \sigma_2 ,o_{euc} \rangle= c \cdot 2n (1 - n^2)r^{1-m},
\end{align}
so choosing the scaling constant $c$ to make $\tilde{\lambda} =1$ finishes the proof. 
\end{proof}
The following result is the key ingredient to our existence theorem 
in the AHE case, which will be proved in Section \ref{completionsec}. 
\begin{prop}
\label{extending}
Let $\sigma = c \cdot \sigma_2 = c \cdot S_{ijkl} z^i z^j dz^k dz^l$ 
be as Proposition \ref{tensork33}. Then 
there exist a solution $\tilde{k}$ of $P_{g_0}(\tilde{k})=0$ defined on $M_0$ 
such that
\begin{align}
\tilde{k}&=\sigma + O(\abs{z}) \text{  as  } \abs{z}\to 0\\
\tilde{k}&=k_{\infty}+o(1)	\text{   as  } x \to 0, 
\end{align}
where $x^{2}k_{\infty}\in \overbar{V_{0}}$.
\end{prop}
\begin{proof}
The proof is immediate from Lemma \ref{biglem} if we let $\sigma$ be the leading term of the tensor $k_{1}$ found in Proposition \ref{tensork33}.
\end{proof}

\section{Completion of proofs}
\label{completionsec}
In this section, we complete the proofs of Theorems \ref{theorem2} and 
\ref{theorem3}. 
First, for $u \in \R$, we define
\begin{align}
h_{t,u}=g_{cal}+th +utk_{1},	
\end{align}
where $h$ was defined in \eqref{eeee}, 
and $k_1$ was defined in Proposition \ref{tensork33}. 
\begin{lem}
\label{htulem}
	For $t$ and $u$  sufficiently small, $h_{t,u}$ is a Riemannian metric on $X^t$. Moreover, for each integer $\ell\ge 0$, there exists a constant $c_{\ell}$ so that $h_{t,u}$ satisfies the following estimate on ${{X}}^{t}$
	\begin{align}
		\abs{\nabla^{\ell} (Ric_{h_{t,u}}-t\Lambda h_{t,u}-t(\lambda+u) \chi_{t} (\rho) o)}_{g_{cal}}\le c_{\ell}t^{2}(1 +u^{2} + t^2)\rho^{2-\ell}.
		\end{align}
\end{lem}
\begin{proof}
	  Since $h$ and $k_{1}$ both have quadratic leading terms, and $g_{cal}$ is ALE, there exist constants $c$ and $c_{1}$  so that 
\begin{align}
\abs{h}_{g_{cal}}(p)\le c\rho^{2}(p)\\
\abs{k_{1}}_{g_{cal}}(p)\le c_{1} \rho^{2}(p),
\end{align}
for any $p\in X^{t}$. Similarly, there exists a constant $c_2$ so that
\begin{align}
\abs{o}_{g_{cal}}(p)\le c_{2} \rho^{-2n}(p),	
\end{align}
for any $p\in X^{t}$. It follows that
\begin{align}
\abs{th}_{g_{cal}}(p)\le c\rho^{2}(p)\le \frac{1}{2}ct^{\frac{1}{2}}\\
\abs{tuk_{1}}_{g_{cal}}(p)\le c_{1} u\rho^{2}(p)\le \frac{1}{2}c_{1}ut^{\frac{1}{2}}\\
\abs{tuo}_{g_{cal}}(p)\le c_{2} u\rho^{2}(p)\le \frac{1}{2}c_{2}ut^{\frac{1}{2}}
\end{align}
The right hand side is independent of $p$  so $h_{t,u}$ is a Riemannian metric for sufficiently small $t$ and $u$. Recall that
\begin{align}
	&d_{g_{cal}}Ric(h)=\Lambda g_{cal}+\lambda o \\
	&d_{g_{cal}}Ric(o)=0\\
	&d_{g_{cal}}Ric(k_{1})=o,
\end{align} 
so we have
\begin{align}
	Ric(g_{cal}+th+tuk_{1})=Ric(g_{cal})+d_{g_{cal}}Ric(th+tuk_{1})+Q(th+tuk_{1}),
\end{align}
where, using \eqref{Rmexp}, $Q$ is a nonlinear term satisfying 
\begin{align}
\begin{split}
\abs{Q(th+tuk_{1})}\le b\Big(
|Rm_{g_{cal}}| \abs{ th + tu k_1}^2 &+
\abs{th+tuk_{1}}\abs{\nabla^{2}(th+tuk_{1})}\\
&+\abs{\nabla (th+tuk_{1})}^{2}\Big),
\end{split}
\end{align}
for some constant $b$ and for $t$ and $u$ sufficiently small, and where the norms and covariant derivatives are measured with respect to the Calabi metric. Note that
\begin{align}
	\abs{\nabla^{\ell}h}\le c_{\ell}\rho^{2-\ell}\\
		\abs{\nabla^{\ell}o}\le c_{\ell}\rho^{-2n-\ell}\\
			\abs{\nabla^{\ell}k_{1}}\le c_{\ell}\rho^{2-\ell},
\end{align}
so it follows that
\begin{align}
\abs{Q(th+tuk_{1})}\le b_{1}t^2 ( 1 + t^{2}+u^{2})\rho^{2}
\end{align}
and consequently
\begin{align}
\abs{Ric_{h_{t,u}}-t\Lambda h_{t,u}-t\lambda o- tu o}\le c_{0}t^2 (1 + t^{2}+u^{2})\rho^{2},
\end{align}
for some constant $c_0$, so the lemma for $\ell=0$ follows. 
The proof for $\ell>0$ is similar and is omitted. 
\end{proof}
We next define 
\begin{align}
k =
\begin{cases}
\tilde{k}  & M_0 \setminus B(p_0, 2 t^{1/4}) \\
(1 - \chi_{t}(r) )\phi_{t}^* \tilde{k} + \chi_{t}(r) t^{2}k_{1} & (1/2) t^{-1/4} < r < 2 t^{-1/4}\\
t^{2} k_{1}  & r < (1/2) t^{-1/4} \\
\end{cases},
\end{align}
where $\chi_{t}$ is the cutoff function we defined in \eqref{cutoff},
and $\tilde{k}$ was defined in Proposition~\ref{extending}. 
For $u\in \R$, consider the metric on $M^t$ defined by 
$g_{t,u} = g_t + u k$, where $g_t$ is the refined metric defined
in \eqref{rmdef}. Note that $g_{t,u}$ has the following expression
\begin{align}
g_{t,u} =
\begin{cases}
g_0+u\tilde{k}  & M_0 \setminus B(p_0, 2 t^{1/4}) \\
(1 - \chi_{t}(r) )\phi_{t}^* (g_0+u\tilde{k}) + \chi_{t}(r) t h_{t,u} & (1/2) t^{-1/4} < r < 2 t^{-1/4}\\
t h_{t,u}  & r < (1/2) t^{-1/4} \\
\end{cases}.
\end{align}
This metric satisfies the following modification of Lemma \ref{estimate2}.
\begin{lem}
\label{udep}
	For $t$ and $u$ sufficiently small, $g_{t,u}$ is a Riemannian metric on $M^t$. 
Moreover, for each integer $\ell\ge 0$, there exist a constant $c_{\ell}$ so that $g_{t,u}$ satisfies the following estimate on ${{X}}^{t}$
	\begin{align}
		\abs{\nabla^{\ell} (Ric_{g_{t,u}}-\Lambda g_{t,u}-t(\lambda+u) \chi_{t} (\rho) o)}_{g_{cal}}\le c_{\ell}t^2(1 + t^{2}+u^{2})\rho^{2-\ell}.
		\end{align}
\end{lem}
\begin{proof}
If $r < (1/2) t^{-1/4}$, then the proof follows from the Lemma \ref{htulem} by scaling.
If 
\begin{align}
\label{bet}
(1/2) t^{-1/4} < r < 2 t^{-1/4},
\end{align}
then 
\begin{align}
g_{t,u}=(1 - \chi_{t}(r))(\phi_{t}^* (u\tilde{k}+g_0)) + \chi_{t}(r) (th_{t}+t^{2}uk_{1}).
\end{align}
We have that
\begin{align}
&\phi_{t}^{\ast}g_0=tg_{euc}+t^{2}H+O(t^{\frac{3}{2}} \rho^3)\\
&th_{t}=tg_{euc}+t^{2}H+tO(r^{-2n})+tO(r^{2-2n+\epsilon}),
\end{align}
and
\begin{align}
&\phi_{t}^{\ast}(u \tilde{k})=t^{2}(u\sigma)+ut^{2}O(r^{-2n+2+\epsilon})\\
&t^{2}uk_{1}=t^{2}u(\sigma)+t^{2}O(r^{-2n+2+\epsilon}),
\end{align}
as $t\rightarrow 0 $. Since the quadratic terms agree in each of these, 
we have the estimate
\begin{align}
&g_{t,u}=tg_{euc}+t^{2}H+t^2 u\sigma+tO(t^{\frac{3}{2}}\rho^{3}),
\end{align}
as $t \to 0$. Writing 
 \begin{align}
 \theta=g_{t}-tg_{euc},	
 \end{align}
expanding the Ricci tensor near $t\cdot g_{euc}$ we have
\begin{align}
&Ric_{g_{t,u}}=d_{t \cdot euc}Ric(\theta)+Q(\theta),
\end{align}
where $Q(\theta)$ satisfies the estimate
\begin{align}
&t^{2}\abs{Q(\theta)}\le c(\abs{\theta}\abs{\nabla^{2}\theta}+\abs{\nabla \theta}^{2}),
\end{align}
for some constant $c$, where the norms are with respect to $g_{euc}$. 
More precisely,
\begin{align}
\theta=t^{2}H+t^{2}u\sigma+tO(t^{\frac{3}{2}}\rho^{3})
\end{align}
implies that
\begin{align}
&\abs{\nabla^{2}\theta}\abs{\theta}=O \big(t^4 (1 + u^2 + t^2)\rho^{2}\big),\\
&\abs{\nabla \theta}^{2}= O\big(t^4(1 + u^2 + t^2) \rho^2\big).
\end{align}
From \eqref{eeee} it follows that 
\begin{align}
	&d_{t\cdot euc}Ric(\theta)=t\Lambda g_{euc}+O\big( t^2 (1 + u^2 + t^2)\rho^{2}\big),
\end{align}
where $O$ is with respect to the Euclidean metric. Note that \eqref{bet} implies that
\begin{align}
t\lambda \chi_{t} (\rho) o=O(t^{2}\rho^{2}).
\end{align}
Finally we obtain
\begin{align}
Ric_{g_{t,u}}-\Lambda g_{t,u}-t(\lambda+u) \chi_{t} (\rho) o=
O\big(t^2(1 + t^{2}+u^{2}\big)\rho^{2}),
\end{align}
so the lemma for $\ell=0$ follows. For $\ell > 0$, the proof is similar and is omitted.
\end{proof}
\begin{proof}[Proof of Theorem \ref{theorem2}]
We apply Proposition \ref{obs}, keeping track of the dependence on 
the variable $u$. From Proposition \ref{obs}, we obtain a solution of 
\begin{align}
\Phi ( \hat{g}_{t,u,v})  =  \lambda(t,u) ( \chi_t o).
\end{align}
Lemma \ref{udep} implies that $\lambda(t,u)$ is differentiable in 
the $u$ variable, and if $\lambda=0$ then 
\begin{align}
\label{lexp2}
\lambda(t,u)=t u+O(t^{\frac{3}{2}- \delta}).	
\end{align}
By the classical implicit function theorem, for each $t$ sufficiently small, 
there exists $u(t)$ such that $\lambda(t,u(t))=0$.
For $t$ sufficiently small, $Ric_{g_{t,u,v}}$ is strictly negative,
and a standard argument (see \cite{MR2260400}) then implies that $\hat{g}_{t,u,v}$  is Einstein.
\end{proof}

\begin{proof}[Proof of Theorem \ref{theorem3}]
First assume that $g_0$ is non-degenerate. 
Observe that the refined approximate metric is close to 
the na\"ive approximate metric in the $C^{2,\alpha}_{\delta_0, \delta_{\infty};t}$
weighted norm, for $t$ sufficiently small. 
As above, from the results in \cite[Section~8]{Biquard}, 
for any metric $g_E$ satisfying \eqref{t1n}, then there exists a 
diffeomorphism $\varphi : M^t \rightarrow M^t$ so that $B_{g_t} ( g_t - \varphi^* g_E) = 0$,
and $\varphi^* g_E$ remains close to $g_t$ in the $C^{2,\alpha}_{\delta_0, \delta_{\infty};t}$ weighted norm.  
The metric $\varphi^* g_E$ would then give a zero of the mapping $\Phi_t$, 
but then \eqref{lexp} would then imply that $\lambda = 0$, a contradiction.

If $g_0$ is not non-degenerate, then if $\delta_{\infty}$ is 
chosen sufficiently small, the linearized operator is
Fredholm (see \cite{Leebook}), so there is a finite-dimensional space 
of $L^2$ infinitesimal Einstein deformations on $(M_0,g_0)$.
The remainder of the proof is almost identical to the 
proof of Theorem \ref{theoremone} given above.
\end{proof}
\bibliographystyle{amsalpha}
\bibliography{Morteza_Viaclovsky}
\end{document}